\documentclass[11pt,final]{amsart}

\usepackage{amssymb,amsmath,latexsym}
\usepackage{amsthm}
\usepackage{amsfonts}
\usepackage{enumerate}
\usepackage{booktabs}
\usepackage{mathrsfs}

\usepackage{hyperref}

\usepackage[color]{showkeys}
\definecolor{refkey}{gray}{.75}
\definecolor{labelkey}{gray}{.75}

\usepackage{graphicx,pgf}
\usepackage{float}

\usepackage{caption}
\usepackage{subcaption}

\usepackage{fancyhdr}

\theoremstyle{plain} 
\newtheorem{thm}{Theorem}[section]
\newtheorem{prop}[thm]{Proposition}
\newtheorem{lemma}[thm]{Lemma}
\newtheorem{cor}[thm]{Corollary}
\newtheorem{example}[thm]{Example}
\newtheorem{rmk}{Remark}

\theoremstyle{definition}

\theoremstyle{remark}

\numberwithin{equation}{section}

\newcommand{\mysectionname}{}
\newcommand{\newsection}[1]{\section{#1}\renewcommand{\mysectionname}{\uppercase{#1}}}

\newcommand{\MT}{\mathcal{M}_{\mathbb{T}}}  

\newcommand{\Mrealplus}{\mathcal{M}_{\mathbb{R}^+}}
\newcommand{\Mrealpplus}{\mathcal{M}_{\mathbb{R}^+}^*}

\newcommand{\g}[1]{\mathit{G}_{#1}}

\newcommand{\pp}[1]{\psi_{#1}}
\newcommand{\et}[1]{\eta_{#1} }

\begin{document}
\title{on the free convolution with a free multiplicative analogue of the normal distribution}
\author{Ping Zhong}
\address{Department of Mathematics, Rawles Hall, 831 East Third Street, Indiana University, Bloomington, Indiana 47405, U.S.A. }
\email{pzhong@indiana.edu}
\begin{abstract}
We obtain a formula for the density of the free convolution of an arbitrary probability measure on the unit circle
of $\mathbb{C}$ with the free multiplicative analogues of the normal distribution on the unit circle. This description relies on a characterization of 
the image of the unit disc under the subordination function, which also allows us to prove some regularity properties
of the measures obtained in this way. As an application, we give a new proof for Biane's classic result on the densities of the 
free multiplicative analogue of the normal distributions. 
We obtain analogue results for probability measures on $\mathbb{R}^+$.
Finally, we describe the density of the 
free multiplicative analogue of the normal distributions as an example and prove unimodality 
and some symmetry properties of 
these measures. 
\end{abstract}

\maketitle
\newsection{Introduction}
Given probability measures $\mu,\nu$ on either the unit circle 
$\mathbb{T} \subset \mathbb{C}$, or
the positive half line $\mathbb{R}^+=[0,+\infty)$, we denote by $\mu\boxtimes \nu$ their free
multiplicative convolution. Free multiplicative Brownian motion is a 
noncommutative stochastic process with stationary increments, which 
have the same distribution as the free multiplicative analogue of the normal distributions.
For $t>0$, we denote by $\lambda_t$ the free multiplicative analogue of the normal distribution on $\mathbb{T}$, 
with $\Sigma$-transform $\Sigma_{\lambda_t}(z)=e^{\frac{t}{2}\frac{1+z}{1-z}}$. Given a probability measure $\mu$ on 
$\mathbb{T}$,
the free Brownian motion with initial distribution $\mu$ is the noncommutative 
stochastic process with distribution $\{\mu\boxtimes\lambda_t: t\geq 0 \}$.
For $t>0$, we denote by $\sigma_t$ the free multiplicative analogue of the normal distribution of the positive half line
with $\Sigma$-transform $\Sigma_{\sigma_t}(z)=e^{\frac{t}{2}\frac{1+z}{z-1}}$. Given a probability measure $\nu$ on 
the positive half line, similarly, 
the free Brownian motion with initial distribution $\nu$ is the noncommutative stochastic process 
with distribution $\{\nu\boxtimes\sigma_t:t\geq 0 \}$.

Biane studied free multiplicative Brownian motion 
in his pioneering papers \cite{BianeBM, BianeJFA}. 
For example, in \cite{BianeBM} he proved that the free unitary Brownian motion
can be approximated by matrix valued Brownian motions, and calculated the moments of $\lambda_t$ and $\sigma_t$ (see also \cite{Rains, Singer}). In 
\cite{BianeJFA} he gave a description of the density of $\lambda_t$ and $\sigma_t$ (see also \cite{DH2011,Zhong1}
for different approaches). 
Voiculescu introduced free liberation processes by using free unitary Brownian motion in his study on free entropy and Fisher information theory \cite{DVV1999}. He also proved some regularity results of the spectral measure of free unitary Brownian motions with arbitrary initial conditions in the same paper. 
These results play important roles in recent works by
Collins and Kemp \cite{CK2014},
Izumi and Ueda \cite{IU2013}, where these authors partially proved so called Unification Conjecture in free entropy and information theory. 
Kemp \cite{2013Kemp1} and 
C{\'e}bron \cite{2013CG}  showed that free positive Brownian can also be approximated by matrix valued Brownian motions (see \cite{2013DHK, 2013Kemp2} for other relevant work). 
In this article we use analytic methods to study the densities of $\mu\boxtimes\lambda_t$ and $\nu\boxtimes\sigma_t$ 
for an arbitrary measure $\mu$ on $\mathbb{T}$
and an arbitrary measure $\nu$ on the positive half line.

We denote 
by $\MT$ the set of probability measures on $\mathbb{T}$,
and by $\mathcal{ID}(\boxtimes, \mathbb{T})$ the set of probability measures on $\mathbb{T}$ which are infinitely divisible with respect
to $\boxtimes$. We also denote by $\Mrealplus$ the set of probability measures on $\mathbb{R}^+$, and 
by $\mathcal{ID}(\boxtimes, \mathbb{R}^+)$ the set of probability measures on $\mathbb{R}^+$ which are infinitely divisible with respect to
$\boxtimes$.

Denote by $P_0$ the Haar measure on $\mathbb{T}$, we have that $P_0\boxtimes\lambda_t =P_0$,
and therefore we restrict ourselves to measures
$\mu\in\MT\backslash\{P_0 \}$. Given $\mu\in\MT\backslash\{P_0\}$, we set $\mu_t=\mu\boxtimes\lambda_t$.
The method which we use to study the density of 
$\mu_t$ relies on a characterization of the image $\Omega_{t, \mu}$ of the unit disc under the subordination map
of $\eta_{\mu_t}$ with respect to $\eta_{\mu}$.
Due to a result in \cite{BB2005}, $\Omega_{t, \mu}$ is simply connected and its boundary is a simple closed curve. We prove that 
the line segment $\{re^{i\theta}:0<r\leq 1\}$ only intersects
$\partial\Omega_{t, \mu}$ at one point for any $\theta\in [0,2\pi)$. We show that the measure $\mu_t$ 
is absolutely continuous with respect to $P_0$, and its
density is analytic whenever it is positive. In addition, we obtain a formula to describe the density of $\mu_t$, and we 
prove that the number of connected components of the support of $\mu_t$ is a non-increasing
function of $t$. The case when $\mu$ is the Dirac measure at $1$ 
yields a new proof for Biane's result on the description of the density of $\lambda_t$ in \cite{BianeJFA}.

We prove similar results for the positive half line. Given $\nu\in\Mrealplus$ which is not
concentrated at 0, set $\nu_t=\nu\boxtimes\sigma_t$. 
Denote by $\Gamma_{t, \nu}$ the image of the upper half plane under the subordination map of 
$\eta_{\nu_t}$ with respect to $\eta_{\nu}$. We show that for every $r>0$ the  semicircle  
$\{re^{i\theta}: 0\leq \theta\leq \pi\}$
intersects $\partial\Gamma_{t, \nu}$ at exactly two points, including one point on the negative half line. 
We give a formula for the density of $\nu_t$, and
prove, among other results, that the number of connected components of 
the support of $\nu_t$ is a non-increasing
function of $t$. When $\nu$ is the Dirac measure at $1$, we obtain a new proof for Biane's \cite{BianeJFA} description
of the density of $\sigma_t$. 
In addition, we prove some symmetry properties of the measure $\sigma_t$.

Our results are multiplicative analogues of results from \cite{Biane1997}, where the free additive
convolution with a semi-circular distribution was studied. We were also influenced by the nice exposition in \cite{CDDF}. 
Our reference for free probability theory is the classical book \cite{Basic} 
by Voiculescu, Dykema and Nica, and our 
reference for the properties of certain subordination functions is Belinschi and Bercovici \cite{BB2005}.
 
In a joint work with H.-W. Huang in \cite{HZ},
we extend our method to study infinitely divisible measures and regularity properties of measures in
partially defined semigroups relative to free multiplicative convolution. 

This article is organized as follows. After this introductory section, we review some preliminaries regarding multiplicative 
free convolution in Section 2. In Section 3, we study the distribution of free multiplicative Brownian motion on $\mathbb{T}$. 
In Section 4, we study the distribution of free multiplicative Brownian motion on $\mathbb{R}^+$.

\newsection{preliminaries}
\subsection{Free multiplicative convolution of measures on \texorpdfstring{$\mathbb{T}$}{}}
For $\mu\in\MT$, we define
\begin{equation}\nonumber
  \pp{\mu}(z)=\int_{\mathbb{T}}\frac{tz}{1-tz}d\mu(t),\, z\in\mathbb{D}
\end{equation}
where $\mathbb{D}$ denotes the open unit disc, and set $\et{\mu}(z)=\pp{\mu}(z)/(1+\pp{\mu}(z))$. 
Then $\et{\mu}(0)=0$ and $\et{\mu}'(0)$ equals the first moment of $\mu$. 
If $\mu$ has nonzero first moment, we define its $\Sigma$-transform by 
    \begin{equation}\nonumber
       \Sigma_{\mu}(z)=\frac{\et{\mu}^{-1}(z)}{z}
    \end{equation} 
for $z$ in a neighborhood of zero. The binary operation $\boxtimes$ on $\MT$, introduced in \cite{DVV1987, BV1992}, represents
the distribution of the product of free unitary random variables. 
If $\mu,\nu\in\MT$ both have nonzero first moment, then $\mu\boxtimes\nu$ is uniquely determined by
      \begin{equation}\nonumber
        \Sigma_{\mu\boxtimes\nu}(z)=\Sigma_{\mu}(z)\Sigma_{\nu}(z)
      \end{equation}
for $z$ in some domain where all three functions involved are defined.       
Given $\mu,\nu\in\MT$, it is proved in \cite[Theorem 3.5]{Biane1998} that
there exists an analytic map $\eta:\mathbb{D}\rightarrow\mathbb{D}$, which we call the
subordination function of $\eta_{\mu\boxtimes\nu}$ with respect to $\eta_{\mu}$,
such that $\eta(0)=0$ and 
\begin{equation}\nonumber
   \eta_{\mu\boxtimes\nu}(z)=\eta_{\mu}(\eta(z))  
\end{equation}
holds for all $z\in\mathbb{D}$. If both $\mu$ and $\nu$ have zero first moment,
then $\mu\boxtimes\nu=P_0$, the Haar measure on $\mathbb{T}$. We are interested
in the case when $\nu$ has nonzero first moment, and in particular $\nu=\lambda_t$. 
As it was pointed out in \cite[Example 3.5]{Zhong1}, when $\mu$ has zero fist moment, 
the subordination function of $\eta_{\mu\boxtimes\nu}$ with respect to $\eta_{\mu}$ is generally not unique.
However, if we require a subordination function satisfying one additional property, then there is 
a unique one, which we call the principal subordination function. 
The next result was first proved in \cite{Biane1998}, and then in \cite{BB2007new} by a different method. In 
\cite{Zhong1}, we proved the uniqueness of the subordination functions 
for the case when one of the measures has zero first moment.
\begin{thm}[\cite{BB2007new, Biane1998}]
 Given $\mu,\nu\in\MT$ such that $\mu$ is different from $P_0$ and
 $\nu$ has nonzero first moment, there exist two unique 
analytic functions $\omega_1,\omega_2:\mathbb{D}\rightarrow\mathbb{D}$ such that
 \begin{enumerate}[$(1)$]
   \item $\omega_1(0)=\omega_2(0)=0;$
   \item $\eta_{\mu\boxtimes\nu}(z)=\eta_{\mu}(\omega_1(z))=\eta_{\nu}(\omega_2(z));$
   \item $\omega_1(z)\omega_2(z)=z\eta_{\mu\boxtimes\nu}(z)$ for all $z\in\mathbb{D}$.
 \end{enumerate} 
\end{thm}
Given $\mu\in\MT\backslash P_0$, we set $\mu_t=\mu\boxtimes\lambda_t$.
We also denote by $\eta_t$ and $\zeta_t$ the unique subordination functions
satisfying the equations 
  \begin{enumerate}[$(1)$]
    \item $\eta_t(0)=\zeta_t(0)=0$,
    \item $\eta_{\mu_t}(z)=\eta_{\mu}(\eta_t(z))=\eta_{\lambda_t}(\zeta_t(z))$ and
    \item $\eta_t(z)\zeta_t(z)=z\eta_{\mu_t}(z)$.
  \end{enumerate} 
By a characterization of the 
$\eta$-transforms of measures on $\mathbb{T}$ in
\cite[Proposition 3.2]{BB2005}, for any $t>0$ there exists a unique measure $\rho_t\in\MT$ such that
   \begin{equation}\label{eq:2.1}
     \eta_t(z)=\eta_{\rho_t}(z)
   \end{equation}
holds for all $z\in\mathbb{D}$. We now recall a result in \cite[Lemma 3.4 and Corollary 3.13]{Zhong1}.
\begin{lemma}[\cite{Zhong1}]\label{lemma:2.2}
 The probability measure $\rho_t$ is $\boxtimes$-infinitely divisible and its $\Sigma$-transform is given by 
      \begin{equation}
        \Sigma_{\rho_t}(z)=\Sigma_{\lambda_t}(\eta_{\mu}(z))=\exp\left(\frac{t}{2}\int_{\mathbb{T}}\frac{1+\xi z}{1-\xi z}d\mu(\xi) \right).
      \end{equation}
\end{lemma}
Let $\Phi_{t, \mu}(z)=z\Sigma_{\rho_t}(z)$ and let $\Omega_{t, \mu}=\{z\in\mathbb{D}:|\Phi_{t, \mu}(z)|<1 \}$, then since 
$\rho_t\in\mathcal{ID}(\boxtimes, \mathbb{T})$, we have that $\Phi_{t, \mu}(\eta_{\rho_t}(z))=z$ for all $z\in \mathbb{D}$.
The function $\Phi_{t, \mu}$ satisfies the properties in \cite[Theorem 4.4, Proposition 4.5]{BB2005}, 
thus the function $\eta_t(z)=\eta_{\rho_t}(z)$ can be extended continuously to the boundary $\partial\mathbb{D}$.
\begin{prop}[\cite{BB2005}]\label{prop:2.3}
 \begin{enumerate}[$(1)$]
   \item The function $\eta_t$ is a conformal map with image $\Omega_{t, \mu}$, and 
   its inverse is the restriction of $\Phi_{t, \mu}$
    to $\Omega_{t, \mu}$. In addition, $\eta_t$ extends to be a continuous function on $\overline{\mathbb{D}}$ 
    and $\eta_t$ is one-to-one on $\overline{\mathbb{D}}$.
   \item $\Omega_{t, \mu}$ is a simply connected domain bounded by a simple closed curve.
   \item If $\xi\in\mathbb{T}$ satisfies $\eta_t(\xi)\in\mathbb{D}$, then $\eta_t$ can be continued analytically to a neighborhood of $\xi$.
 \end{enumerate}
\end{prop}

We close this section with a formula which allows us to recover $\mu$ from its $\eta$-transform. 
For $\mu\in\MT$, we have that
   \begin{equation}\label{eq:2.46}
      \frac{1}{2\pi}\left(\frac{1+\et{\mu}(z)}{1-\et{\mu}(z)}\right)=\frac{1}{2\pi}
       \int^{2\pi}_{0}\frac{e^{i\theta}+z}{e^{i\theta}-z}d\mu(e^{-i\theta}), \,z\in\mathbb{D}.
   \end{equation}
The real part of this function is the Poisson integral of the measure $d\mu(e^{-i\theta})$,
and then $\mu$ can be recovered from (\ref{eq:2.46}) by the Stieltjes inversion formula. The functions
   \begin{equation}\label{eq:2.5}
     \frac{1}{2\pi}\Re{\left(\frac{1+\et{\mu}(re^{i\theta})}{1-\et{\mu}(re^{i\theta})} \right)}
       =\frac{1}{2\pi}\frac{1-|\et{\mu}(re^{i\theta})|^2}{|1-\et{\mu}(re^{i\theta})|^2}
   \end{equation}
converge to the density of $d\mu(e^{-i\theta})$ a.e. relative to Lebesgue measure $d\theta$, and they converge to infinity a.e. relative to the singular part of this measure.
\subsection{Free multiplicative convolution of measures on \texorpdfstring{$\mathbb{R}^+$}{}}
Given $\mu\in\Mrealplus$, we define
\begin{equation}\nonumber
  \psi_{\mu}(z)=\int_0^{+\infty}\frac{tz}{1-tz}d\mu(t),\,z\in \mathbb{C}\backslash \mathbb{R}^+
\end{equation}
and set $\et{\mu}(z)=\pp{\mu}(z)/(1+\pp{\mu}(z))$. 
The case when $\mu=\delta_0$, the 
Dirac measure at $0$, is trivial, we thus set
$\Mrealpplus=\mathcal{M}_{\mathbb{R}^+}\backslash\{\delta_0\}$.
Given $\mu\in\Mrealpplus$,
it is shown in \cite{BV1993} that the function $\et{\mu}$
is univalent in the left half plane $i\mathbb{C}^+$. 
The $\Sigma$-transform of $\mu$ is defined by
   \begin{equation}\nonumber
      \Sigma_{\mu}(z)=\frac{\et{\mu}^{-1}(z)}{z}
   \end{equation}
for $z$ in $\et{\mu}(i\mathbb{C}^+)$. Given $\mu,\nu\in\Mrealpplus$, 
the free multiplicative convolution of $\mu$ and $\nu$, denoted by $\mu\boxtimes\nu$,
is uniquely determined by 
    \begin{equation}
       \Sigma_{\mu\boxtimes\nu}(z)=\Sigma_{\mu}(z)\Sigma_{\nu}(z)
    \end{equation}
in some domain where all three functions involved are defined.
It is also known from \cite{BB2007new, Biane1998} that there exist two analytic functions, which
are called subordination functions,
$\omega_1,\omega_2:\mathbb{C}\backslash\mathbb{R}^+ \rightarrow \mathbb{C}\backslash\mathbb{R}^+$
such that
\begin{enumerate}[$(1)$]
  \item $\omega_j(0-)=0$ for $j=1,2$;
  \item for any $\lambda\in\mathbb{C}^+$, we have $\omega_j(\overline{\lambda})=
      \overline{\omega_j(\lambda)}$, $\omega_j(\lambda)\in\mathbb{C}^+$ 
      and $\arg\omega_j(\lambda)\geq \arg\lambda$ for $j=1,2$;
  \item $\et{\mu\boxtimes\nu}(z)=\eta_{\mu}\left(\omega_1(z) \right)
        =\eta_{\nu}\left(\omega_2(z)\right)$ for $z\in\mathbb{C}\backslash\mathbb{R}^+$.
\end{enumerate}

Fix $\nu\in\Mrealpplus$ and let $\omega_t$ be the subordination function of $\eta_{\nu\boxtimes\sigma_t}$
with respect to $\eta_{\nu}$, that is $\eta_{\nu\boxtimes\sigma_t}(z)=\eta_{\nu}(\omega_t(z))$. 
Due to \cite[Proposition 6.1]{BV1993}, there exists a measure $\tau_t\in\Mrealpplus$, such
that $\omega_t(z)=\eta_{\tau_t}(z)$ for $z\in\mathbb{C}\backslash\mathbb{R}^+$. We then recall
\cite[Proposition 4.3]{Zhong1} as follows. 
\begin{prop}
 The measure $\tau_t$ is $\boxtimes$-infinitely divisible and its $\Sigma$-transform is
given by
    \begin{equation}\label{eq:2.6}
      \Sigma_{\tau_t}(z)=\Sigma_{\sigma_t}\left(\eta_{\nu}(z) \right)
       =\exp \left( \frac{t}{2} \int_0^{+\infty}\frac{1+\xi z}{\xi z -1}d\nu(\xi) \right).
    \end{equation}
\end{prop}

Let $H_{t,\nu}(z)=z\Sigma_{\tau_t}(z)$, and let $\Gamma_{t,\nu}$ be the connected component of the 
set $\{z\in\mathbb{C}^+:\Im(H_{t,\nu}(z))>0 \}$ whose boundary contains the negative half line $(-\infty,0)$.
We then have that $H_{t,\nu}(\omega_t(z))=z$ for $z\in \mathbb{C}^+$
and that $\omega_t(H_{t,\nu}(z))=z$ for $z\in\Gamma_{t,\nu}$.
The following result deduces that this definition of $\Gamma_{t,\nu}$ agrees with 
the one given in the introduction.
\begin{prop}\label{prop:2.6}
  \begin{enumerate}[$(1)$]
    \item The restriction of $\omega_t$ to $\mathbb{C}^+$ is a conformal map with image $\Gamma_{t,\nu}$, and its
     inverse is the restriction of $H_{t,\nu}$ to $\Gamma_{t,\nu}$. In addition, $\omega_t$ extends to be a 
     continuous function on $\mathbb{C}^+\cup \mathbb{R}$, and $\omega_t$ is one to one on this set.
    \item If $\xi\in(0,+\infty)$ satisfies $\Im(\omega_t(\xi))>0$, then $\omega_t$ can be continued
    analytically to a neighborhood of $\xi$. 
  \end{enumerate}
\end{prop}
\begin{proof}
Since $\omega_t$ is the $\eta$-transform of $\tau_t$ and $\tau_t\in \mathcal{ID}(\boxtimes,\mathbb{R}^+)$,
we have that $\omega_t=\eta_{\tau_{t/2}\boxtimes\tau_{t/2}}$. Then by \cite[Proposition 5.2]{BB2005} and
its proof, we see that $\omega_t$ extended continuous to $\mathbb{C}^+\cup\mathbb{R}\backslash\{0\}$,
and the extended function is analytic at points $\xi\in(0,+\infty)$ at which the 
extended function is not real. This proves (2) and part of (1).  

Recall that $\omega_t$ is the subordination function of $\eta_{\nu\boxtimes\sigma_t}$ with respect to $\eta_{\nu}$,
from a general result concerning subordination functions for free multiplicative convolution of 
arbitrary measures on $\mathbb{R}^+$ in
\cite[Remark 3.3]{Belinschi06}, we deduce that $\omega_t$ extends continuously to $0$ as well if $\nu$
is not a Dirac measure at one point. Note that $\lambda_t$ is compactly supported, thus 
$\tau_t$ is compactly supported as well if $\nu$ is a Dirac measure at one point, we conclude that 
$\omega_t$ also extends continuously to $0$ for this case. This finishes the proof.
\end{proof}

\newsection{The unit circle case}
In this section, we prove some regularity properties of $\Omega_{t, \mu}$ and then prove our main results
regarding distributions of free multiplicative Brownian motion on $\mathbb{T}$. 
We use polar coordinates in our discussion and parametrize $\mathbb{T}$ 
as $\mathbb{T}=\{e^{i\theta}: -\pi\leq \theta\leq \pi\}$. 
For fixed $t>0$, define
   \begin{equation}\nonumber
      h_t(r,\theta)=1-\frac{t}{2}\frac{1-r^2}{-\ln r}\int_{\mathbb{T}}\frac{1}{|1-re^{i\theta} \xi|^2}d\mu(\xi).
   \end{equation}
We have
    \begin{equation}\label{eq:3.4}
       \begin{split}
      \ln (|\Phi_{t, \mu}(re^{i\theta})|)&=\ln r+ \frac{t}{2}\int_{\mathbb{T}}\frac{1-r^2}{|1-re^{i\theta} \xi|^2}d\mu(\xi)\\
           &=(\ln r)h_t(r,\theta).
      \end{split}
    \end{equation}
    
To study the boundary of $\Omega_{t,\mu}$, we need the following lemma.
    \begin{lemma}\label{tech}
  Given $-1\leq y\leq 1$, define a function of $r$ by
     \begin{equation}\nonumber
        T_{y}(r)=\frac{1-r^2}{-\ln r}\frac{1}{1-2ry+r^2}
     \end{equation}
on the interval $(0,1)$,
then $T_{y}'(r)>0$ for all $r\in(0,1)$.
\end{lemma}
\begin{proof}
The substitution $x=-\ln r$ implies that it suffices to prove $f'(x)<0$ for $x\in(0,+\infty)$, where
     \begin{equation}\nonumber
       f(x)=\frac{1-e^{-2x}}{x}\frac{1}{1-2e^{-x}y+e^{-2x}}.
     \end{equation}
We have
     \begin{equation}\nonumber
        \begin{split}
       f'(x)
           = \frac{(-e^{4x}+4xe^{2x}+1)+y(-2xe^{3x}+2e^{3x}-2xe^x-2e^x)}{[x(e^{2x}-2e^x y+1)]^2}.
        \end{split}
     \end{equation}
By Taylor expansion, we can check that $-2x e^{3x}+2e^{3x}-2x e^x-2e^x <0$ for all $x>0$; thus, to prove $f'(x)<0$,
it is enough to show that
       \begin{equation}\label{eq:3.3}
          2x e^{3x}-2e^{3x}+2x e^x+2e^x < e^{4x}-4xe^{2x}-1.
       \end{equation}
It is easy to check that (\ref{eq:3.3}) holds for all $x\in (0,+\infty)$ by calculating 
the Taylor expansions in both sides of (\ref{eq:3.3}).
\end{proof}

Lemma \ref{tech} implies that the function $h_t(r,\theta)$ is a decreasing function of $r$ on (0,1) for fixed $\theta$. 
Define $h_t:[-\pi,\pi]\rightarrow \mathbb{R}\cup\{-\infty\}$ as follows:
        \begin{equation}\nonumber
          \begin{split}
             h_t(\theta)&=\lim_{r\rightarrow 1^{-}}h_t(r,\theta)\\
                  &=1-t\int_{-\pi}^{\pi}\frac{1}{|1-e^{i(\theta+x)}|^2}d\mu(e^{ix})\\
                  & =1-t\int_{-\pi}^{\pi}\frac{1}{2-2\cos(\theta+x)}d\mu(e^{ix}).
          \end{split}
        \end{equation}
We now let 
    \begin{equation}\nonumber
      \begin{split}
         U_{t, \mu}&=\left\{ -\pi\leq \theta\leq \pi : h_t(\theta)<0 \right\}\\
            &=\left\{-\pi\leq\theta\leq\pi : \int_{-\pi}^{\pi}\frac{1}{|1-e^{i(\theta+x)}|^2}d\mu(e^{ix})>\frac{1}{t}\right\}
      \end{split}
    \end{equation}
and $U_{t, \mu}^c=[-\pi,\pi]\backslash U_{t, \mu}$. We also define a function 
$v_t: [-\pi,\pi]\rightarrow (0,1]$ as
      \begin{equation}
        \begin{split}\label{eq:3.0001}
           v_t(\theta)&=\sup\left\{0< r<1: h_t(r,\theta)> 0 \right\}\\
                &=\sup\left\{0< r<1: \frac{1-r^2}{-2\ln r}\int_{-\pi}^{\pi}\frac{1}{|1-re^{i(\theta+x)}|^2}d\mu(e^{ix}) 
                < \frac{1}{t}\right\}.
        \end{split}
      \end{equation}

The next result regarding regularity of $\Omega_{t, \mu}$ is fundamental to our discussion. Note that
this type of result was proved in \cite{Zhong1} in the study of the density of $\lambda_t$. 
\begin{thm}\label{thm:3.2}
The sets $\Omega_{t, \mu}$ and $\partial\Omega_{t, \mu}$ can be characterized by the functions $h_t(r,\theta)$ and $v_t(\theta)$ as follows.
  \begin{enumerate}[$(1)$]
    \item $\Omega_{t, \mu}=\{re^{i\theta}: 0\leq r<v_t(\theta), \theta\in[-\pi,\pi]\}$.
    \item $\partial\Omega_{t, \mu}\cap \mathbb{D}=
    \{re^{i\theta}: \theta\in U_{t, \mu}, \text{and}\,\, h_t(r,\theta)=0 \}$.
  \end{enumerate}
\end{thm}
\begin{proof}
From Proposition \ref{prop:2.3} and (\ref{eq:3.4}), we see that
      \begin{equation}\nonumber
          \begin{split}
         \Omega_{t, \mu}\cap\mathbb{D}
                   =\{re^{i\theta}: 0\leq r<1,-\pi\leq \theta\leq \pi \,\text{and}\, h_t(r,\theta)>0\}.
          \end{split}         
      \end{equation}
Since $h_{t,\theta}(r):=h_t(r,\theta)$ is a decreasing function of $r$, we then deduce the following assertions:
     \begin{enumerate}[(i)]
       \item If $\theta\in U_{t, \mu}$, then $\{re^{i\theta}:0\leq r<v_t(\theta)\} \subset \Omega_{t, \mu} $ 
           and $v_t(\theta)e^{i\theta}\in\partial\Omega_{t, \mu}$.
       \item If $\theta\in U_{t, \mu}^c$, then $\{re^{i\theta}:0\leq r<1\}\subset \Omega_{t, \mu}$ and $e^{i\theta}\in\partial\Omega_{t, \mu}$.
     \end{enumerate}
Notice that for $\theta\in U_{t, \mu}^c$, we have that $v_t(\theta)=1$, we thus proved (1) and (2). 
\end{proof}

\begin{cor}
 The function $\eta_{\mu}$ has a continuous extension to $\overline{\Omega_{t, \mu}}$, and 
 the function $\eta_{\mu_t}$ has a continuous extension to $\overline{\mathbb{D}}$.
\end{cor}
\begin{proof}
 From Lemma \ref{lemma:2.2}, we see that 
     \begin{equation}\label{eq:2.3}
      \Sigma_{\rho_t}(\eta_{\rho_t}(z))= \Sigma_{\lambda_t}(\eta_{\mu}(\eta_{\rho_t}(z))).
     \end{equation}
The identity $\eta_{\rho_t}(z)\Sigma_{\rho_t}(\eta_{\rho_t}(z))=\Phi_{t, \mu}(\eta_{\rho_t}(z))=z$ and (\ref{eq:2.3}) imply that
      \begin{equation}\label{eq:2.4}
         \frac{z}{\eta_{\rho_t}(z)}=\exp\left(\frac{t}{2}\frac{1+\eta_{\mu}(\eta_{\rho_t}(z))}
         {1-\eta_{\mu}(\eta_{\rho_t}(z))} \right)
      \end{equation} 
holds for $z\in\mathbb{D}$.
The identity $\Phi_{t, \mu}(\eta_{\rho_t}(z))=z$ also implies that the function $\eta_{\rho_t}$ is never zero in $\mathbb{D}\backslash\{0\}$,
and thus $z/\eta_{\rho_t}(z)$ is never zero 
on $\mathbb{D}$. 
Since $\eta_{\rho_t}=\eta_t$,
it follows from Proposition \ref{prop:2.3} that 
$\eta_{\rho_t}$ extends continuously to $\overline{\mathbb{D}}$. 

Let $h(z)=\frac{t}{2}\frac{1+z}{1-z}$, then by the definition of $\eta$-transform, we have 
\[
 h(\eta_{\mu}(z))=\frac{t}{2}\int_{\mathbb{T}}\frac{1+z\xi}{1-z\xi}d\mu(\xi).
\]
Note that the exponent in the right hand side of (\ref{eq:2.4}) equals
$h(\eta_{\mu}(\eta_{\rho_t}(z)))$. Given $z\in\mathbb{D}$, we 
set $\eta_{\rho_t}(z)=re^{i\theta}$, where $r<v_t(\theta)\leq 1$ according to Theorem \ref{thm:3.2}. We then have that
\begin{equation}
  \begin{split}
   |\Im h(\eta_{\mu}(\eta_{\rho_t}(z)))|&=|\Im h(\eta_{\mu}(re^{i\theta}))| \\
                 &=\left|t\int_{-\pi}^{\pi}\frac{r\sin(\theta+x)}{|1-re^{i(\theta+x)}|^2}d\mu(e^{ix})   \right|\\
                 &\leq \frac{2r\ln r}{r^2-1}\leq 1,
  \end{split}   
\end{equation}
where we used the definition of $v_t$ in (\ref{eq:3.0001}).

Thus we can take logarithms in both sides of (\ref{eq:2.4}).
The assertion follows easily by continuous extension and the fact that
$\eta_{\rho_t}(\overline{\mathbb{D}})=\overline{\Omega_{t, \mu}}$.
\end{proof}

\begin{lemma}\label{lemma:3.3}
 The support of $\mu$ is contained in the closure of the set $U_{t, \mu}^{-1} $ which is defined by
        \begin{equation}\nonumber
           U_{t, \mu}^{-1} =\left\{e^{-i\theta}: \theta\in U_{t, \mu}\right\}.
        \end{equation}
\end{lemma}
\begin{proof}
   Let $\theta_0 \in [-\pi,\pi]\backslash \overline{U_{t, \mu}}$, where $\overline{U_{t, \mu}}$ is the closure of $U_{t, \mu}$.
Without loss of generality, we may assume that $\theta_0$ is different from $\pi$ and $-\pi$,
then there is some $\epsilon >0$, such that 
     \begin{equation}\nonumber
         [\theta_0-\epsilon,\theta_0+\epsilon]\subset (-\pi,\pi)\backslash\overline{U_{t, \mu}}.
     \end{equation}
For any integer $n\geq 2$, we define $\alpha_k=\theta_0-\epsilon+2k\epsilon/n$ for all $0\leq k \leq n$, then
the sets $[\alpha_k,\alpha_{k+1}]$ are contained in $(-\pi,\pi)\backslash \overline{U_{t, \mu}}$. 
Given $\theta\in[\alpha_k,\alpha_{k+1}]$, then we have that
      \begin{equation}\nonumber
         \begin{split}
            \frac{1}{t}&\geq \int_{-\alpha_{k+1}}^{-\alpha_k}\frac{1}{|1-e^{i(\theta+x)}|^2}d\mu(e^{ix})\\
                    &=\int_{-\alpha_{k+1}}^{-\alpha_k}\frac{1}{|e^{i\theta}-e^{-ix}|^2}d\mu(e^{ix})\\
                    &\geq\frac{\mu(\{e^{i\theta}: -\alpha_{k+1}\leq \theta\leq-\alpha_k\})}{|e^{-i\alpha_k}-e^{-i\alpha_{k+1}}|^2}.
         \end{split}
      \end{equation}
This implies that
         \begin{equation}\nonumber
           \begin{split}
              \mu(\{e^{i\theta}: -(\theta_0+\epsilon)\leq \theta \leq -(\theta_0-\epsilon) \})
                &\leq \sum_{k=0}^{n-1} \mu(\{e^{i\theta}: -\alpha_{k+1}\leq \theta\leq-\alpha_k\})\\
                &\leq\frac{1}{t}\sum_{k=0}^{n-1}|e^{-i\alpha_k}-e^{-i\alpha_{k+1}} |^2.
           \end{split}
         \end{equation}
Notice that $\sum_{k=0}^{n-1}|e^{-i\alpha_k}-e^{-i\alpha_{k+1}} |^2$ can be arbitrarily small if we choose
$n$ large enough, therefore 
    \begin{equation}\nonumber
      \mu(\{e^{i\theta}: -(\theta_0+\epsilon)\leq \theta \leq -(\theta_0-\epsilon) \})=0
    \end{equation}
and our assertion follows.
\end{proof}
\begin{prop}\label{prop:3.4}
 For any interval $ (\alpha,\beta)\subset U_{t, \mu}^c $, the function 
  $f(\theta)=\int_{-\pi}^{\pi}\frac{1}{|1-e^{i(\theta+x)}|^2}d\mu(e^{ix})$
   is strictly convex in $(\alpha,\beta)$.
\end{prop}
\begin{proof}
We notice that 
  \begin{equation}\nonumber
    \int_{-\pi}^{\pi}\frac{1}{|1-e^{i(\theta+x)}|^2}d\mu(e^{ix})=
      \int_{-\pi}^{\pi}\frac{1}{|e^{ix}-e^{-i\theta}|^2}d\mu(e^{ix})=\int_{-\pi}^{\pi}\frac{1}{2-2\cos(\theta+x)}d\mu(e^{ix}),
  \end{equation}
from Lemma \ref{lemma:3.3}, we deduce that $f$ is analytic on $(\alpha, \beta)$. We calculate its 
second derivative
   \begin{equation}\nonumber
      f''(\theta)=\frac{1}{2}\int_{-\pi}^{\pi}\frac{1-\cos(\theta+x)+\sin^2(\theta+x) }{(1-\cos(\theta+x))^3 }d\mu(e^{ix})>0,
   \end{equation}
which implies the strict convexity of $f$.
\end{proof}
\begin{rmk}
\emph{
Proposition \ref{prop:3.4} implies that $\int_{-\pi}^{\pi}\frac{1}{|1-e^{i(\theta+x)}|^2}d\mu(e^{ix})<1/t$
for $\theta$ in the interior of $U_{t, \mu}^c$. }
\end{rmk}

We are now ready to discuss the density of the measure $\mu_t=\mu\boxtimes\lambda_t$. 
\begin{prop}
 The measure $\mu_t$ is absolutely continuous with respect to Lebesgue measure. Moreover,
the density is analytic whenever it is positive. 
\end{prop}
\begin{proof}
By \cite[Lemma 5.1]{Zhong1}, $\eta_{\lambda_t}(\overline{\mathbb{D}})$ does not contain $1$. Since
$\eta_{\mu_t}$ is subordinated with respect to $\eta_{\lambda_t}$, $\eta_{\mu_t}(\overline{\mathbb{D}})$ does
not contain $1$ either. Then from (\ref{eq:2.5}), we see that 
$\Re\left(\frac{1+\eta_{\mu_t}(e^{i\theta})}{1-\eta_{\mu_t}(e^{i\theta})}  \right)$ is bounded, which implies
that $\mu_t$ has no singular part
with respect to 
Lebesgue measure, and thus $\mu_t$ is absolutely continuous with respect to 
Lebesgue measure. 

We rewrite (\ref{eq:2.4}) as 
    \begin{equation}\nonumber
         \frac{z}{\eta_t(z)}=\exp\left(\frac{t}{2}\frac{1+\eta_{\mu_t(z)}}{1-\eta_{\mu_t(z)} }\right).
      \end{equation}
If the density of $\mu_t$ is positive at $e^{-i\theta}$, then 
by (\ref{eq:2.5}), we have $|\eta_{\mu_t}(e^{i\theta})|<1$, which implies that $|\eta_t(e^{i\theta})|<1$.
The analyticity of the density follows from part (3) of Proposition \ref{prop:2.3}.
\end{proof}
Define a map $\Psi_{t, \mu}:\mathbb{T}\rightarrow \mathbb{T}$ by
   \begin{equation}\nonumber
      \Psi_{t, \mu}(e^{i\theta})=\Phi_{t, \mu}(v_t(\theta)e^{i\theta}).
   \end{equation}
Recall that $|\Phi_{t, \mu}(v_t(\theta)e^{i\theta})|=1$, we calculate 
    \begin{equation}\nonumber
      \arg(\Phi_{t, \mu}(v_t(\theta)e^{i\theta}))=\theta+t\int_{-\pi}^{\pi}
      \frac{v_t(\theta)\sin(\theta+x)}{|1-v_t(\theta)e^{i(\theta+x)}|^2}d\mu(e^{ix}).
    \end{equation}
\begin{prop}
 The map $e^{i\theta}\rightarrow v_t(\theta)e^{i\theta}$ is a homeomorphism from $\mathbb{T}$ onto $\partial\Omega_{t, \mu}$;
and the map $\Psi_{t, \mu}$ is a homeomorphism from $\mathbb{T}$ onto $\mathbb{T}$.
\end{prop}
\begin{proof}
Only continuity of the functions needs to prove. 
The function $h_t(r,\theta)$ is continuous on $[0,1)\times [-\pi,\pi]$
and $h_t(r,\pi)=h_t(r,-\pi)$. It follows easily from
the definition of $v_t$ that the function $v_t$ is continuous
on $[-\pi,\pi]$, which yields the first part. The function
$\Phi_{t,\mu}$ has a continuous extension to $\overline{\Omega_{t,\mu}}$ by Proposition \ref{prop:2.3},
thus the map $\Psi_{t, \mu}$ is continuous.
\end{proof}

\begin{thm}\label{thm:3.7}
The measure $\mu_t$ has a density given by 
         \begin{equation}\label{eq:3.17}
          p_t(\overline{\Psi_{t, \mu} (e^{i\theta})}) =-\frac{\ln v_t(\theta)}{\pi t}.
         \end{equation}
\end{thm}
\begin{proof}
We prove it by a direct calculation. 
     \begin{equation}\nonumber
         \begin{split}
        p_t(\overline{\Psi_{t, \mu}(e^{i\theta})})&=p_t(\overline{\Phi_{t, \mu}(v_t(\theta)e^{i\theta})}) \\  
          &\stackrel{\text{\tiny(1)}}{=}\frac{1}{2\pi}\Re\left(\frac{1+\eta_{\mu}(v_t(\theta)e^{i\theta}) }{1-\eta_{\mu}(v_t(\theta)e^{i\theta}) } \right)\\
          &\stackrel{\text{\tiny(2)}}{=}
          \frac{1}{2\pi}\Re \left(\int_{\mathbb{T}}
          \frac{1+\xi v_t(\theta)e^{i\theta}}{1-\xi v_t(\theta)e^{i\theta}}d\mu(\xi) \right)   \\
          &=\frac{1}{2\pi}\int_{-\pi}^{\pi}\frac{1-v_t(\theta)^2}{|1-v_t(\theta)e^{i(\theta+x)}|^2}d\mu(e^{ix})\\
            &\stackrel{\text{\tiny(3)}}{=}
            \frac{1}{2\pi}\frac{-2\ln v_t(\theta)}{t}=\frac{-\ln v_t(\theta)}{\pi t}.         
          \end{split}
     \end{equation}
The equality $(1)$ is due to (\ref{eq:2.5}) and
the fact $\eta_t(\Phi_{t, \mu}(z))=z$ for $z\in\overline{\Omega_{t, \mu}}$;
the equality $(2)$ follows from (\ref{eq:2.46}); and
the equality $(3)$ follows from the definition of $v_t(\theta)$.
\end{proof}
\begin{cor}\label{cor:3.8}
Let $A_t$ be the support of the measure $\mu_t$, and let $A_t^{-1}=\{e^{-i\theta}:e^{i\theta}\in A_t \} $,
then $A_t^{-1}$ equals to the image of the closure of $U_{t, \mu}$ by the homeomorphism $\Psi_{t, \mu}$. The number of 
connected components of the support of $\mu_t$ is a non-increasing function of $t$. 
\end{cor}
\begin{proof}
The map $\Psi_{t,\mu}$ is a homeomorphism from $\mathbb{T}$ onto $\mathbb{T}$.
By Theorem \ref{thm:3.7}, we have that $p_t(\overline{\Psi_{t, \mu} (e^{i\theta})})>0$
if and only if $v_t(\theta)\neq 1$.
By the definitions of $U_{t,\mu}$ and $v_t(\theta)$, we have 
        \begin{equation}\nonumber
           \begin{split}
              \{\theta:v_t(\theta)\neq 1 \}
              &=\{\theta: \exists \,\, 0\leq r<1 \,\, \text{so that}\,\, h_t(r,\theta)=0\}\\
                            &=\{\theta:h_t(\theta)<0 \}\\
                            &=U_{t,\mu}.
           \end{split}
        \end{equation}
Therefore, the support of $\mu_t$ is the conjugate of the set $\Psi_{t,\mu}(U_{t,\mu})$,
which implies the first assertion.        

From Theorem \ref{thm:3.7}, to prove the second assertion, it is enough to prove that the number of 
connected components of $U_{t, \mu}$ is a non-increasing function of $t$, which is 
a consequence of Proposition \ref{prop:3.4}.
\end{proof}
\begin{prop}
The number $v_t(\theta)$ is bounded below by $a_t$, where $a_t$ is the unique solution of the equation
    \begin{equation}\nonumber
      r\exp\left(\frac{t}{2}\frac{1+r}{1-r} \right)=1;
    \end{equation}
and thus the density of $\mu_t$ is bounded above by $-\ln a_t/\pi t$. Moreover,
the density of $\mu_t$ tends to $1/2\pi$ uniformly as $t\rightarrow \infty$; in particular,
the support of $\mu_t$ is $\mathbb{T}$ when $t$ is sufficiently large.
\end{prop}
\begin{proof}
We notice that 
     \begin{equation}\label{eq:3.21}
       \frac{1-r^2}{-2\ln r}\int_{-\pi}^{\pi}\frac{1}{|1-re^{i(\theta+x)}|^2}d\mu(e^{ix})\leq \frac{1+r}{(-2\ln r)(1-r)}.
     \end{equation}
By taking $y=1$ for Lemma \ref{tech}, we see that the right hand side of (\ref{eq:3.21}) is also an
increasing function of $r$, which yields that $v_t(\theta)\geq a_t$. 
The last assertion is a consequence of \cite[Corollary 3.27]{Zhong1}.
\end{proof}
\begin{rmk}
\emph{From the description of the density of $\lambda_t$ in Example \ref{example:3.10} below, or from
\cite[Theorem 5.4]{Zhong1}, we see that $-\ln a_t/\pi t$ is the maximum
of the density of $\lambda_t$.
}
\end{rmk}

As an example, we now apply Theorem \ref{thm:3.7} to give a description of the density of $\lambda_t$
which recovers Biane's classic results regarding the density of $\lambda_t$ proved in \cite{BianeJFA}. 
See also \cite{DH2011, Zhong1} for different approaches. We provide
a picture of the density function of $\lambda_t$ for some values of $t$ in the end of this paper.
\begin{example}\label{example:3.10}
\emph{
Let $\mu=\delta_1$, the Dirac measure at $1$, we denote $\Omega_t=\Omega_{t,\delta_1}$, $\Phi_t=\Phi_{t,\delta_1}$,
$\Psi_t=\Psi_{t,\delta_1}$ and $U_t=U_{t,\delta_1}$.
Then we have that $\mu_t=\lambda_t$ and 
   \begin{equation}\nonumber
       \begin{split}
         U_t=\left\{-\pi\leq\theta\leq \pi: \frac{1}{|1-e^{i\theta}|^2}>\frac{1}{t} \right\}=
         \left\{-\pi\leq \theta\leq \pi: 1-\cos\theta<\frac{t}{2}\right \}
       \end{split}
   \end{equation}
which yields that $U_t=\mathbb{T}$ if $t>4$, and 
     \begin{equation}\nonumber
       U_t=(-\arccos(1-t/2), \arccos(1-t/2))
     \end{equation}
if $t\leq 4$. If $\theta\in U_t$, then $v_t(\theta)$ satisfies the equation
       \begin{equation}\label{eq:3.24}
         \frac{1-v_t(\theta)^2}{-2\ln v_t(\theta)}\frac{1}{1-2v_t(\theta)\cos\theta+v_t(\theta)^2}=\frac{1}{t},
       \end{equation}
and we have 
     \begin{equation}\label{eq:3.251}
       \arg(\Psi_t(e^{i\theta}))=\theta+t\frac{v_t(\theta)\sin\theta}{1-2v_t(\theta)\cos\theta+v_t(\theta)^2}. 
     \end{equation}
For $t\leq 4$, let $\theta_0(t)=\arccos(1-t/2)$, then $v_t(\theta_0(t))=1$, from which we deduce that
       \begin{equation}\label{eq:3.26}
         \arg(\Psi_t(e^{i\theta_0(t)}))=\theta_0(t)+\sin\theta_0(t).
       \end{equation}
From Corollary \ref{cor:3.8} and (\ref{eq:3.26}), we see that the support of $\lambda_t$ is the set 
       \begin{equation}\nonumber
         \left\{e^{i\theta}: -\arccos\left(1-\frac{t}{2}\right)-\frac{1}{2}\sqrt{(4-t)t}\leq\theta
              \leq \arccos\left(1-\frac{t}{2}\right)+\frac{1}{2}\sqrt{(4-t)t} \right\}
       \end{equation} 
if $t<4$; and the support of $\lambda_t$ is $\mathbb{T}$ if $t\geq 4$.}
       
\emph{ By Theorem \ref{thm:3.7} and (\ref{eq:3.24}), we obtain that the density of $\lambda_t$ at 
       $\Phi_t(v_t(\theta)e^{i\theta})$ is given by 
            \begin{equation}\label{eq:3.28}
              \frac{1}{2\pi}\frac{1-v_t(\theta)^2}{1-2v_t(\theta)\cos\theta+v_t(\theta)}
                 =\frac{1}{2\pi}\Re\left(\frac{1+v_t(\theta)e^{i\theta}}{1-v_t(\theta)e^{i\theta}} \right),
            \end{equation}
where we used the fact that $v_t(-\theta)=v_t(\theta)$ due to the identity
$\Phi_t(\overline{z})=\overline{\Phi_t(z)}$.
Note that $\Phi_t(z)=z\exp(\frac{t}{2}\frac{1+z}{1-z})$,
then 
(\ref{eq:3.28}) implies that, for $\theta\in U_t$, at the point $\omega=\Phi_t(v_t(\theta)e^{i\theta})$, the density is positive 
   and is equal to the product of $1/2\pi$ and the real part of $k(t,\omega)$, 
   where $k(t,\omega)$ is the only solution, with positive real part,
   of the equation 
       \begin{equation}\nonumber
          \frac{z-1}{z+1}e^{\frac{t}{2}z}=\omega.
       \end{equation}
}
\end{example}

Theorem \ref{thm:3.2} and \cite[Lemma 5.1]{Zhong1} allow us to obtain a better 
description for $\Omega_t$, which implies the unimodality of $\lambda_t$ as it was 
shown in \cite[Theorem 5.4]{Zhong1}.
\begin{prop}
For any $t>0$, there exists a non-decreasing function $\gamma_t:[0,\pi]\rightarrow (0,1]$ such that
     \begin{equation}\nonumber
       \partial\Omega_t=\{\gamma_t(\theta)e^{i\theta}: 0\leq \theta\leq \pi\}\cup  
                      \{\gamma_t(\theta)e^{-i\theta}: 0\leq \theta\leq \pi\}. 
     \end{equation}
\end{prop}
\begin{proof}
For given $0<r<1$, we notice that the map $\theta \rightarrow |\Phi_{t,\delta_1}(re^{i\theta})|=
r\exp\left(\frac{t}{2}\frac{1-r^2}{1-2r\cos\theta+r^2}\right)$ is 
a strictly decreasing function of $\theta$ on the interval $[0,\pi]$, which implies that
if $z=r_0e^{i\theta_0}\in\partial\Omega_t\cap\mathbb{D}\cap\mathbb{C}^+$
for $\theta_0\in [0,\pi)$,
then we have that 
    \begin{equation}\label{eq:3.05}
       \{re^{i\theta}:r=r_0, \theta_0<\theta \leq\pi \}\subset \Omega_t.
    \end{equation}
Note that $\Omega_t$ is symmetric with respect to $x$-axis, then (\ref{eq:3.05}) and
Theorem \ref{thm:3.2} yield our assertion. 
\end{proof}
\begin{cor}
The density of $\lambda_t$ is symmetric with respect to $x$-axis; and the density
of $\lambda_t$ has a unique local maximum at $1$.
\end{cor}

\begin{rmk}\label{rmk:order1}
\emph{
For $0<t<4$, 
let $\ln v_t(\theta)=-x$, we have
\[
 \frac{1-v_t(\theta)^2}{-2\ln v_t(\theta)}\approx 1-x+\frac{2x^2}{3}+o(x^2)
\]
for $\theta$ close to $\theta_0(t)$. We also have
\[
 \frac{1-2v_t(\theta)\cos\theta+v_t(\theta)^2}{t}\approx 1
    -x+\frac{2\sin\theta_0(t)(\theta-\theta_0(t))}{t}+o(|\theta-\theta_0(t)|)
\]
for $\theta$ close to $\theta_0(t)$,
where we used the fact $1-\cos\theta_0(t)=t/2$. 
From (\ref{eq:3.24}), we thus deduce that there exists $a(t)>0$ such that
     \begin{equation}\label{eq:3.40}
        -\ln v_t(\theta)=a(t)|\theta_0(t)-\theta|^{\frac{1}{2}}(1+o(1))
     \end{equation}
in a small interval $(\theta_0(t)-\epsilon,\theta_0]$ for some $\epsilon>0$. 
By taking derivative for (\ref{eq:3.251}) and noticing
(\ref{eq:3.40}), we derive that
      \begin{equation}\nonumber
         \frac{d}{d\theta}\arg(\Psi_t(e^{i\theta}))\bigg|_{\theta=\theta_0}>0. 
      \end{equation}
Set $\theta(t)=\arccos\left(1-t/2\right)+\sqrt{(4-t)t}/2$ be one end point of the support of $\lambda_t$,
then by using
Theorem \ref{thm:3.7} we deduce that there exists $b(t)>0$ such that
     \begin{equation}\label{eq:3.41}
      \frac{d\lambda_t}{d\theta}\left(e^{i\theta}\right)=b(t)|\theta-\theta(t)|^{\frac{1}{2}}(1+o(1))
      \end{equation}
in a small interval $(\theta(t)-\epsilon',\theta(t)]$ for some $\epsilon'>0$.
}

\emph{Let us focus on the case that $t=4$. Let $v_4(\theta)=1-x(\theta)$, from 
(\ref{eq:3.24}), by applying a similar calculation for the case $0<t<4$, we can easily obtain that $x(\theta)=O(|\pi-\theta|)$ for
$\theta$ in a small interval $(\pi-\epsilon',\pi]$. 
From (\ref{eq:3.251}), we can calculate that
       \begin{equation}\nonumber
         \begin{split}
           \arg(\Psi_4(e^{i\theta}))&=O(x(\theta)^2)|\pi-\theta|\\
                  &=O(|\pi-\theta|^3)
         \end{split}
       \end{equation}
for $\theta$ in a small interval $(\pi-\epsilon',\pi]$.}
\emph{By applying Theorem \ref{thm:3.7}, we see that there exists $b(4)>0$ such that
       \begin{equation}\label{eq:3.42}
          \frac{d\lambda_4}{d\theta}\left(e^{i\theta} \right)=b(4)\left|\pi-\theta\right|^{\frac{1}{3}}(1+o(1))
       \end{equation}
in a small interval $(\pi-\epsilon',\pi]$.
}

\emph{We note that the function $\Phi_t$ has zeros of order one at $e^{i\theta_0(t)}$ and $e^{-i\theta_0(t)}$ for $0<t<4$ and
$\Phi_4$ has a zero of order two at $-1$.
The orders in (\ref{eq:3.41}) and (\ref{eq:3.42}) are essentially due to this fact.
}
\end{rmk}

\section{The positive half line case}
In this section, we prove some useful properties of $\Gamma_{t,\nu}$, and give a description
of the density of $\nu_t=\nu\boxtimes\sigma_t$.

We use polar coordinates and the parametrization $\mathbb{C}^+=\{re^{i\theta}:0<r<+\infty, 0<\theta<\pi\}$.
For $z=re^{i\theta}\in\mathbb{C}^+$, we calculate 
    \begin{equation}\nonumber
         |H_{t,\nu}(z)|=r\exp\left(\frac{t}{2}\int_0^{+\infty}\frac{r^2\xi^2-1}{1+r^2\xi^2-2r\xi\cos\theta}d\nu(\xi) \right)
    \end{equation}
and 
    \begin{equation}\label{eq:4.1}
         \begin{split}
       \arg(H_{t,\nu}(z))&=\theta-t\int_0^{+\infty}\frac{r\xi\sin\theta}{1+r^2\xi^2-2r\xi\cos\theta}d\nu(\xi)\\
                   &=\theta\left[1-\frac{t (\sin\theta)}{\theta}
                   \int_0^{+\infty}\frac{r\xi}{1+r^2\xi^2-2r\xi\cos\theta}d\nu(\xi) \right].
         \end{split}
    \end{equation}
For $0<r<+\infty$ and $0<\theta<\pi$, we set 
     \begin{equation}\nonumber
       f_t(r,\theta)=1-\frac{t (\sin\theta)}{\theta}
                   \int_0^{+\infty}\frac{r\xi}{1+r^2\xi^2-2r\xi\cos\theta}d\nu(\xi).
     \end{equation}
\begin{lemma}
For fixed $0<r<+\infty$, the function $f_t(r,\theta)$ is an increasing function of $\theta$ on $(0,\pi)$.
\end{lemma}
\begin{proof}
We calculate 
        \begin{equation}\nonumber
           \left(\frac{\sin\theta}{\theta} \right)'=\frac{\cos\theta}{\theta^2}(\theta-\tan\theta)<0,
        \end{equation}
and 
       \begin{equation}\nonumber
         \frac{d}{d\theta}\left(\int_0^{+\infty}\frac{r\xi}{1+r^2\xi^2-2r\xi\cos\theta}d\nu(\xi)\right)
            =\int_0^{+\infty}\frac{-2r^2\xi^2\sin\theta}{(1+r^2\xi^2-2r\xi\cos\theta)^2}d\nu(\xi)<0
       \end{equation}
for all $\theta\in(0,\pi)$ and $0<r<\infty$. The assertion follows from these calculations.
\end{proof}

We set $f_t(r)=\lim_{\theta\rightarrow 0^+}f_t(r,\theta)$, then we have that 
       \begin{equation}\nonumber
         f_t(r)=1-t\int_0^{+\infty}\frac{r\xi}{(1-r\xi)^2}d\nu(\xi).
       \end{equation}
Let   \begin{equation}\nonumber
          \begin{split}
        V_{t,\nu}&=\{0<r<+\infty: f_t(r)<0 \}\\
             &=\left\{0<r<+\infty: \int_0^{+\infty}\frac{r\xi}{(1-r\xi)^2}d\nu(\xi)>\frac{1}{t} \right \}.
          \end{split}   
      \end{equation}
We also define a function $u_t:(0,+\infty)\rightarrow [0,\pi)$ as
      \begin{equation}\label{eq:4.0001}
          \begin{split}
         u_t(r)&=\inf\{0\leq \theta<\pi: f_t(r,\theta)\geq 0\}\\
              &=\inf\left\{0\leq \theta<\pi: \frac{\sin\theta}{\theta}
                   \int_0^{+\infty}\frac{r\xi}{1+r^2\xi^2-2r\xi\cos\theta}d\nu(\xi) \leq \frac{1}{t}\right\}.
           \end{split} 
      \end{equation}
\begin{rmk}\label{rmk:3}
\emph{
If $r\in V_{t,\nu}$, then $u_t(r)>0$ and we have
     \begin{equation}\nonumber
        \frac{\sin(u_t(r))}{u_t(r)}
                   \int_0^{+\infty}\frac{r\xi}{1+r^2\xi^2-2r\xi\cos(u_t(r))}d\nu(\xi)= \frac{1}{t}.
     \end{equation} 
The function $u_t$ depends on $\nu$, and we always fix an
arbitrary $\nu\in\Mrealpplus$ in this article when we discuss $u_t$. 
Only in Proposition \ref{prop:4.12} and Lemma \ref{lemma:4.13}, we choose $\nu$ as $\nu=\delta_1$.}
\end{rmk}
\begin{thm}\label{thm:4.2}
The sets $\Gamma_{t,\nu},\partial\Gamma_{t,\nu}$ can be characterized by the functions $f_t(r,\theta)$
and $u_t(\theta)$ as follows.
  \begin{enumerate}[$(1)$]
    \item $\Gamma_{t,\nu}=\{re^{i\theta}:0<r<+\infty, u_t(r)<\theta<\pi \}$.
    \item $\partial\Gamma_{t,\nu}\cap\mathbb{C}^+=\{re^{i\theta}: r\in V_{t,\nu},\text{and}\, f_t(r,\theta)=0\}$.
    \item  We have that \begin{equation}\nonumber
        \begin{split}
          \partial\Gamma_{t,\nu}\cap (0,+\infty)&=\{r:0<r<+\infty, u_t(r)=0 \}=(0,+\infty)\backslash V_{t,\nu}\\
                    &=\left\{0<r<+\infty: \int_0^{+\infty}\frac{r\xi}{(1-r\xi)^2}d\nu(\xi)\leq \frac{1}{t}\right\}.
        \end{split}
      \end{equation}
  \end{enumerate}
\end{thm}
\begin{proof}
 Since $\arg(H_{t,\nu})=\theta \cdot f_t(r,\theta)$, and the function $f_{t,r}(\theta):=f_t(r,\theta)$
 is an increasing function of $\theta$ on $(0,\pi)$, notice that $\lim_{\theta\rightarrow\pi}f_{t,r}(\theta)=1>0$,
 then from (\ref{eq:4.1}), we deduce that 
    \begin{equation}\nonumber
       \Gamma_{t,\nu}=\{z=re^{i\theta}:0<r<+\infty, f_t(r,\theta)>0 \},
    \end{equation} 
which yields (1) and (2). Part (3) follows from the definition of $u_t$ and (2).
\end{proof}

From (\ref{eq:2.6}), we have the following useful formula.
 \begin{equation}\label{eq:2.7}
     \begin{split}
   \frac{z}{\omega_t(z)}&=\Sigma_{\tau_t}(\omega_t(z))
             =\Sigma_{\sigma_t}(\eta_{\nu}(\omega_t(z)))\\
               &=\exp\left(\frac{t}{2}\frac{\eta_{\nu_t}(z)+1}{\eta_{\nu_t}(z)-1} \right),
     \end{split}
 \end{equation}
which yields the next result. 
\begin{prop}
 The function $\eta_{\nu}$ has a continuous extension to $\overline{\Gamma_{t,\nu}}\backslash\{0\}$,
 and the function $\eta_{\nu_t}$ has a continuous extension to $\mathbb{C}^+\cup(\mathbb{R}\backslash\{0\})$.
\end{prop}
\begin{proof}
We first note that $\omega_t(0)=0$ and $\omega_t(z)\neq 0$ for all 
$z\in \mathbb{C}^+\cup(\mathbb{R}\backslash\{0\})$. 

Given $z\in\mathbb{C}^+$, we let $\omega_t(z)=re^{i\theta}$ where $u_t(r)< \theta<\pi$ according
to Theorem \ref{thm:4.2}.
Let $k(z)=\frac{t}{2}\frac{1+z}{z-1}$. By the definition of $\eta$-transform, we have
\[
   k(\eta_{\nu}(z))=\frac{t}{2}\int^{+\infty}_0\frac{1+\xi z}{\xi z-1}d\nu(\xi).
\]
Note that $k(\eta_{\nu}(\omega_t(z)))$ equals the exponent involving in (\ref{eq:2.7}). We thus have
\begin{equation}
   \begin{split}
   \Im k(\eta_{\nu}(\omega_t(z)))&=-t\int^{+\infty}_0\frac{r\xi \sin\theta}{1+r^2\xi^2-2r\xi\cos\theta}d\nu(\xi)\\
     &\geq -\theta \\
     &\geq -\pi,
   \end{split}
\end{equation}
where we used the definition of $u_t(r)$ in (\ref{eq:4.0001}).
It is clear that $\Im k(\eta_{\nu}(\omega_t(z)))\leq 0$.
Thus we can take logarithms on
both sides of (\ref{eq:2.7}). By applying Proposition \ref{prop:2.6}, we 
deduce our assertion. 
\end{proof}

\begin{lemma}\label{lemma:4.3}
The support of $\nu$ is contained in the closure of the set $V_{t,\nu}^{-1}$, which is defined by
   \begin{equation}\nonumber
      V_{t,\nu}^{-1}=\left\{\frac{1}{x}:x\in V_{t,\nu}\right\}.
   \end{equation}
\end{lemma}
\begin{proof}
Let $x_0\in(0,+\infty)\backslash \overline{V_{t,\nu}}$, and choose $\epsilon>0$ sufficiently small such that
$[x_0-\epsilon, x_0+\epsilon]\subset (0,+\infty)\backslash \overline{V_{t,\nu}}$, it is enough to show that
$\nu$ has no charge on the interval $[1/(x_0+\epsilon), 1/(x_0-\epsilon)]$. For any integer $n\geq 1$, we define 
$\beta_k=x_0-\epsilon+2k\epsilon/n$ for all $0\leq k \leq n$. Then since $[\beta_k,\beta_{k+1}]\subset
(0,+\infty)\backslash \overline{V_{t,\nu}}$, then for any $r\in [\beta_k,\beta_{k+1}]$, we have that
    \begin{equation}\nonumber
      \begin{split}
        \frac{1}{t}&\geq \int_{1/\beta_{k+1}}^{1/\beta_k}\frac{r\xi}{(1-r\xi)^2}d\nu(\xi) \\
              &\geq \frac{1}{(x_0+\epsilon)^2 m_k}\nu([1/\beta_{k+1},1/\beta_k]),
      \end{split}
    \end{equation}
where $m_k=\max\{|1-\beta_k/\beta_{k+1}|^2, |1-\beta_{k+1}/\beta_k|^2 \}$.
This yields that 
     \begin{equation}\nonumber
        \nu([1/\beta_{k+1},1/\beta_k])\leq \frac{(x_0+\epsilon)^2}{t}m_k.
     \end{equation}
Notice that $\sum_{k=0}^{n-1} m_k=o(1/n)$ for large $n$, thus we have $\nu([1/(x_0+\epsilon), 1/(x_0-\epsilon)])=0$.
\end{proof}
\begin{prop}\label{prop:4.4}
For any interval $(a,b)\subset (0,+\infty)\backslash \overline{V_{t,\nu}}$, the function 
$g(r)=\int_0^{+\infty}\frac{r\xi}{(1-r\xi)^2}d\nu(\xi)$
is strictly convex on $(a,b)$.
\end{prop}
\begin{proof}
We first note that $\nu$ has no charge on the interval $(1/b,1/a)$ by Lemma \ref{lemma:4.3}, thus 
$g$ is analytic on the interval $(a,b)$. The second derivative of $g$ is 
        \begin{equation}\nonumber
          g''(r)=\int_0^{+\infty}\frac{4\xi^2+2r\xi^2}{(1-r\xi)^4}d\nu(\xi)>0,
        \end{equation}
which yields the desired conclusion.
\end{proof}
Proposition \ref{prop:4.4} yields that $\int_0^{+\infty}\frac{r\xi}{(1-r\xi)^2}d\nu(\xi)<1/t$ for $r$
in the interior of $(0,+\infty)\backslash V_{t,\nu}$.

We now define a map $\Lambda_{t,\nu}:(0,+\infty)\rightarrow (0,+\infty)$ by 
      \begin{equation}\nonumber
         \Lambda_{t,\nu}(r)=H_{t,\nu}(re^{i u_t(r)}).
      \end{equation}
Note that $\arg(H_{t,\nu}(re^{i u_t(r)}))=0$, we then have 
      \begin{equation}\nonumber
          \Lambda_{t,\nu}(r)=r\exp\left(\frac{t}{2}\int_0^{+\infty}\frac{r^2\xi^2-1}{|1-r\xi e^{i u_t(r)}|^2}d\nu(\xi) \right).
      \end{equation}
\begin{prop}
The map $r\rightarrow re^{iu_t(r)}$ is a homeomorphism from $(0,+\infty)$ onto 
\[\partial\Gamma_{t,\nu}\backslash (-\infty,0].\] 
The map $\Lambda_{t,\nu}$ is a homeomorphism from $(0,+\infty)$ onto itself.
\end{prop}  
\begin{proof}
It is enough to prove that the function $u_t$ is continuous. 
By the continuity of the function $f_t(r,\theta)$, it is easy to see that
$u_t$ is continuous at $r_0$ if $u_t(r_0)>0$ for $r_0\in\mathbb{R}^+$.
By the results concerning the boundary of $\Gamma_{t,\nu}$ (Proposition \ref{prop:2.6} and
Theorem \ref{thm:4.2}), the set $\{r\in\mathbb{R}^+:u_t(r)=0 \}$ consisting
of a union of closed intervals. Therefore, $u_t$ is continuous on $\mathbb{R}^+$.
\end{proof}

Define the Cauchy transform of the probability measure $\nu$ by 
    \begin{equation}\nonumber
      \g{\nu}(z)=\int_{-\infty}^{+\infty}\frac{1}{z-t}d\nu(t), \, z\in\mathbb{C}\backslash \mathbb{R},
    \end{equation}
we then have 
   \begin{equation}\label{eq:4.7}
      \g{\nu}\left(\frac{1}{z} \right)=\frac{z}{1-\eta_{\nu}(z)},
   \end{equation}
thus the density of the measures can be recovered by (\ref{eq:4.7}) from
the Stieltjes inverse formula. We notice that $1/(1-\eta_{\nu}(z))=1+\psi_{\nu}(z)$,
which implies that 
   \begin{equation}\nonumber
      \begin{split}
     \Im\left(\frac{1}{1-\eta_{\nu}(z)} \right)&=\Im\left(\int_0^{+\infty}\frac{1}{1-z\xi}d\nu(\xi) \right)\\
                     &=\int_0^{+\infty}\frac{r\xi \sin\theta}{1+r^2\xi^2-2r\xi\cos\theta}d\nu(\xi),
      \end{split}
   \end{equation}
where $z=re^{i\theta}$.
From Remark \ref{rmk:3}, for $z=re^{iu_t(r)}$, we have that
       \begin{equation}\label{eq:4.8}
          \Im\left( \frac{1}{1-\eta_{\nu}\left(re^{iu_t(r)}\right)} \right) = \frac{u_t(r)}{t}.
       \end{equation}
Equations (\ref{eq:4.7}) and (\ref{eq:4.8}), together with the Stieltjes inverse formula yield the following theorem.
\begin{thm}\label{thm:4.6}
  The measure $\nu_t$ has a density given by
      \begin{equation}\nonumber
         q_t \left(\frac{1}{\Lambda_{t,\nu}(r)} \right)=\Lambda_{t,\nu}(r)\frac{1}{\pi}\frac{u_t(r)}{t}
      \end{equation}
for $r\in (0,+\infty)$.
\end{thm}
\begin{proof}
It is a direct calculation. 
    \begin{equation}\nonumber
       \begin{split}
          q_t \left(\frac{1}{\Lambda_{t,\nu}(r)} \right)&=q_t\left(\frac{1}{H_{t,\nu}(re^{iu_t(r)})} \right)\\
                      &=\frac{1}{\pi}\Im\left(\g{\nu_t}\left(\frac{1}{H_{t,\nu}(re^{iu_t(r)})} \right) \right)\\
                      &=\frac{1}{\pi}H_{t,\nu}(re^{iu_t(r)}) \Im\left( \frac{1}{1-\eta_{\nu_t}(H_{t,\nu}(re^{iu_t(r)}))} \right)\\
                      &=\frac{1}{\pi}H_{t,\nu}(re^{iu_t(r)})\Im\left( \frac{1}{1-\eta_{\nu}\left(re^{iu_t(r)}\right)} \right)\\
                      &=\Lambda_{t,\nu}(r)\frac{1}{\pi}\frac{u_t(r)}{t}.
       \end{split}
    \end{equation}
\end{proof}
\begin{cor}
Let $B_t$ be the support of the measure $\nu_t$, and let $B_t^{-1}=\{1/x: x\in B_t\}$, then $B_t^{-1}$ equals
to the image of the closure of $V_{t,\nu}$ by the homeomorphism $\Lambda_{t,\nu}$. The number of 
connected components of the support of $\nu_t$ is a non-increasing function of $t$.
\end{cor}
\begin{proof}
By the definitions of $V_{t,\nu}$ and $u_t(r)$, we have 
      \begin{equation}\label{eq:4.44}
          \begin{split}
             \{r:u_t(r)>0 \}&=\{r: \exists\,\,0<\theta<\pi \,\,\text{so that}\,\,f_t(r,\theta)=0 \} \\
                      &=\{r:f_t(r)<0\}\\
                      &=V_{t,\nu}.
          \end{split}
      \end{equation}
Theorem \ref{thm:4.6} and the fact that 
       \begin{equation}\nonumber
         u_t(r)> 0\Leftrightarrow  r\in V_{t,\nu}
       \end{equation}
yield the first assertion. To prove the second assertion, it is enough
to show that the number of connected components of $V_{t,\nu}$ is a non-increasing function
of $t$, which follows from Proposition \ref{prop:4.4}.
\end{proof}
The following lemma was pointed out in \cite{BB2005}.
\begin{lemma}[\cite{BB2005}]\label{lemma:100}
 Given a probability measure $\mu$ on $[0,+\infty)$, the point $x=0$ is an atom if and only if 
 $\lim_{x\downarrow  -\infty}\eta_{\mu}(x)$ is finite and the value of the limit is 
       \begin{equation}\nonumber
          \lim_{x\downarrow  -\infty}\eta_{\mu}(x)=1-\frac{1}{\mu(\{0\})}.
       \end{equation}
\end{lemma}
\begin{prop}
The measure $\nu_t$ is absolutely continuous with respect to Lebesgue measure and its  
density $q_t$ is analytic on the set $\{x\in (0,+\infty):q_t(x)>0\}$.
\end{prop}
\begin{proof} 
We first note that $\eta_{\sigma_t}(\overline{\mathbb{C}^+})$ does not contain $1$ (see Example \ref{example:4.10}). 
There exists a subordination function $\omega_{t}':\mathbb{C}^+\rightarrow\mathbb{C}^+$
such that $\eta_{\nu_t}(z)=\eta_{\nu\boxtimes\sigma_t}(z)=\eta_{\sigma_t}(\omega_{t}'(z))$,
as we know in Section 2.2.
We thus deduce that $\eta_{\nu_t}(\mathbb{C}^+\backslash\{0\})$ does 
contain $1$ as well. From (\ref{eq:4.7}), we then see that the nontangential limit of $\g{\nu_t}$ is bounded at any 
$x\in (0,+\infty)$, which implies that $\nu_t$ has no singular part in $(0,+\infty)$. 

If the density of $\nu_t$ is positive at $1/x$, then from (\ref{eq:4.7}), we deduce that $\eta_{\nu_t}(x)\in \mathbb{C}^+$.
Thus the identity
  \begin{equation}\label{eq:4.901}
     \frac{z}{\omega_t(z)}=\exp\left(\frac{t}{2}\frac{\eta_{\nu_t}(z)+1}{\eta_{\nu_t}(z)-1} \right)
  \end{equation}
implies that $\omega_t(x)\in\mathbb{C}^+$, and then the analyticity of the density function follows from the fact that
$\omega_t$ can be extended analytically to a neighborhood of $x$.

It remains to show that $\nu_t(\{ 0\})=0$. Notice that $H_{t,\nu}(\omega_t(x))=x$ for $x\in(-\infty,0)$ and
$\lim_{x\downarrow -\infty}H_{t,\nu}(x)=-\infty$, we then have that $\lim_{x\downarrow -\infty}\omega_t(x)=-\infty$.
Thus we have that
      \begin{equation}\nonumber
         \lim_{x\downarrow -\infty}\frac{x}{\omega_t(x)}=\lim_{x\downarrow -\infty}\Sigma_{\tau_t}(\omega_t(x))
         =\exp\left(\frac{t}{2}\right),
      \end{equation}
which also implies that $\lim_{x\downarrow -\infty}\eta_{\nu_t}(x)=-\infty$
by (\ref{eq:4.901}).
From Lemma \ref{lemma:100}, we deduce that $\nu_t(\{0\})=0$.
\end{proof}

\begin{cor}
 If $\nu$ is compactly supported on $(0,+\infty)$, then the support of $\nu_t$ 
has only one connected component for large $t$.
\end{cor}
\begin{proof}
Fix $t=t_0$, then we choose a closed interval $[a,b]\subset (0,+\infty)$ containing 
$V_{t_0,\nu}$. The function $g(r)=\int_0^{\infty}\frac{r\xi}{(1-r\xi)^2}d\nu(\xi)$ has 
a positive minimum on the compact set $[a,b]\backslash V_{t_0,\nu}$.
The number of connected components of $V_{t,\nu}$ is a non-increasing function of $t$
and we have $V_{t_1,\nu}\subset V_{t_2,\nu}$ for any $t_1<t_2$. 
By the definition of $V_{t,\nu}$, we then derive that 
$V_{t,\nu}$ has only one connected component when $t$ is large enough.
Our assertion follows from
 the fact that 
       \begin{equation}\nonumber
         u_t(r)> 0\Leftrightarrow  r\in V_{t,\nu}
       \end{equation}
and Theorem \ref{thm:4.6}.
\end{proof}

Finally, we apply results in this section to give a description of the density of $\sigma_t$,
which was first obtained in \cite{BianeJFA}. See also 
\cite{Zhong1} for a different approach. We provide a picture
of the density function of $\sigma_t$ in the end of this paper.
\begin{example}\label{example:4.10}
\emph{Let $\nu=\delta_1$, we denote $ \Gamma_t=\Gamma_{t,\delta_1}, \Lambda_t=\Lambda_{t,\delta_1}$
and $V_t=V_{t,\delta_1}$, we also set $H_t(z)=H_{t,\delta_1}(z)=z\exp\left(\frac{t}{2}\frac{z+1}{z-1} \right)$. 
Then we have that $\nu_t=\sigma_t$ and 
     \begin{equation}\label{eq:4.125}
        V_t=\left\{0<r<+\infty: \frac{r}{(1-r)^2}>\frac{1}{t}\right\}=(x_1(t),x_2(t)),
     \end{equation}
where $x_1(t)=(2+t-\sqrt{t(t+4)})/2$ and $x_2(t)=(2+t+\sqrt{t(t+4)})/2$.
Since $1\in V_t$ and $u_t(r)>0$ for all $r\in V_t$, we see that $\eta_{\sigma_t}(\mathbb{C}^+)$
does not contain $1$ by applying Theorem \ref{thm:4.2}.}

\emph{
If $r\in V_t$, then $u_t(r)$ satisfies the equation
    \begin{equation}\label{eq:4.13}
      \frac{\sin(u_t(r))}{u_t(r)}\frac{r}{1+r^2-2r\cos(u_t(r))}=\frac{1}{t},
    \end{equation}
and we have
        \begin{equation}\label{eq:4.14}
          \frac{r\sin(u_t(r))}{1+r^2-2r\cos(u_t(r))}=\Im\left(\frac{1}{1-re^{iu_t(r)}} \right).
        \end{equation}
We also have 
      \begin{equation}\nonumber
        \Lambda_t(r)=r\exp\left(\frac{t}{2}\frac{r^2-1}{|1-re^{iu_t(r)}|^2} \right), 
      \end{equation}
and in particular, when $u_t(r)=0$, we have $\Lambda_t(r)=r\exp\left(\frac{t}{2}\frac{r+1}{r-1} \right)$.
We calculate
       \begin{equation}\nonumber
      \begin{split}
              x_3(t):=\frac{1}{\Lambda_t(x_2(t))}
           =\frac{2+t-\sqrt{t(t+4)}}{2}\exp\left(-\frac{\sqrt{t(t+4)}}{2}\right).
      \end{split}
    \end{equation}
and
        \begin{equation}\nonumber
      \begin{split}
              x_4(t):=\frac{1}{\Lambda_t(x_1(t))}
           =\frac{2+t+\sqrt{t(t+4)}}{2}\exp\left(\frac{\sqrt{t(t+4)}}{2}\right),
      \end{split}
    \end{equation}
where we used the fact that $u_t(x_1(t))=u_t(x_2(t))=0$.
We remark that $x_1(t)x_2(t)=x_3(t)x_4(t)=1$.
}
\end{example}
Note that the formulas we use are different from the formulas used in \cite{BianeJFA} by Biane, 
we then record our description 
of the density of $\sigma_t$ as follows.
\begin{prop}\label{prop:4.11}
The support of the measure $\sigma_t$ is equal to the interval $C_t=[x_3(t),x_4(t)]$. The density is positive
on the interior of the interval $C_t$, and is equal to, at the point $x\in C_t$, to 
$(1/(\pi x))\Im(l(t,x))$, where $l(t,x)$ is the unique solution of the equation
   \begin{equation}\nonumber
      \frac{z}{z-1}\exp\left(t\left(z-\frac{1}{2}\right) \right)=x
   \end{equation}
on $\Gamma_t$.
\end{prop}
\begin{proof}
 By the definition of $x_3(t), x_4(t)$ and Theorem \ref{thm:4.6}, we see that
 $C_t=[x_3(t),x_4(t)]$. 
 Let $r=\Lambda_t^{-1}(1/x)$, then by Theorem \ref{thm:4.6}, the density at the point $x$
 is equal to $(1/(\pi x))u_t(r)/t$. By (\ref{eq:4.13}) and (\ref{eq:4.14}), 
         \begin{equation}
           \frac{u_t(r)}{t}=\Im\left(\frac{1}{1-re^{iu_t(r)}} \right),
         \end{equation}
we then set $l(t,x)=1/(1-re^{iu_t(r)})$. Since $\Lambda_t(r)=1/x$, we then have that 
 $H_t(re^{iu_t(r)})=1/x$. The description of $l(t,x)$ follows from this identity and the definition of $H_t$.
\end{proof}

We now turn to discuss symmetry of the measure $\sigma_t$. 
\begin{prop}\label{prop:4.12}
For $\nu=\delta_1$, $u_t$ is a strictly increasing function of $r$ on the interval $(x_1(t),1]$,
$u_t$ is a strictly decreasing function of $r$ on the interval $[1,x_2(t))$, and $u_t(r)=0$ for
all $r\in (0,+\infty)\backslash (x_1(t),x_2(t))$. In particular, $u_t$ attains its global maximum at $1$.
\end{prop}
\begin{proof}
Let $g$ be the function defined by 
     \begin{equation}\nonumber
        g(r,\theta)=\frac{\sin\theta}{\theta}\frac{r}{1+r^2-2r\cos\theta}
     \end{equation}
on the set $(0,+\infty)\times [0,\pi)$. The first part of the assertion follows from (\ref{eq:4.13}) and
the following fact: 
      \begin{enumerate}[(i)]
        \item $\frac{\partial g(r,\theta)}{\partial r}<0$, for $(r,\theta)\in (1,+\infty)\times [0,\pi)$.
        \item $\frac{\partial g(r,\theta)}{\partial r}>0$, for $(r,\theta)\in (0,1)\times [0,\pi)$.
        \item $\frac{\partial g(r,\theta)}{\partial \theta}<0$, for $(r,\theta)\in (0,+\infty)\times (0,\pi)$. 
      \end{enumerate}
The rest part follows from (\ref{eq:4.125}).
\end{proof}

\begin{lemma}\label{lemma:4.13}
For $\nu=\delta_1$, we have 
     \begin{equation}\nonumber
        u_t(r)=u_t\left( \frac{1}{r}\right), \text{and}\, \Lambda_t(r)\cdot\Lambda_t\left(\frac{1}{r}\right)=1,
     \end{equation}
holds for all $r>0$.
\end{lemma}
\begin{proof}
We first note that $r\in V_t=(x_1(t),x_2(t))$ if and only if $1/r\in V_t$.
If $r\notin V_t$, then $u_t(r)=0$, and we have
     \begin{equation}\nonumber
       \Lambda_t(r)=r\exp\left(\frac{t}{2}\frac{r+1}{r-1} \right)\, \text{and}\,
         \Lambda_t\left( \frac{1}{r}\right) =\frac{1}{r}\exp\left(\frac{t}{2}\frac{r+1}{1-r} \right),
     \end{equation}
which prove the case when $r\notin V_t$.
     
If $r\in V_t$, we first prove that $u_t(r)=u_t(1/r)$. Recall from (\ref{eq:4.13}) that the
pair $(r,u_t(r))$ satisfies the equation
     \begin{equation}\nonumber
        \frac{\sin\theta}{\theta}\frac{r}{1+r^2-2r\cos\theta}=\frac{1}{t}
     \end{equation}
which is equivalent to 
     \begin{equation}\label{eq:4.17}
        r^2-\left(2\cos\theta+t\frac{\sin\theta}{\theta} \right)r+1=0.
     \end{equation}
For any $\theta$, (\ref{eq:4.17}) is a quadratic equation of $r$ and the product of 
its solutions is $1$. This observation and Proposition \ref{prop:4.12} imply that
$u_t(r)=u_t(1/r)$. For any $\theta\in (0,u_t(1)]$, let $r_1, r_2$ be the solutions
of (\ref{eq:4.17}), then $r_1, r_2\in (0,+\infty)$ and $r_1\cdot r_2=1$. We claim
that $\Lambda_t(r_1)\cdot\Lambda_t(r_2)=1$. If fact, $\theta=u_t(r_1)=u_t(r_2)$, and
   \begin{equation}\label{eq:4.19}
      \ln\left [\Lambda_t(r_i) \right]=\ln r_i+\frac{t}{2}\frac{r_i^2-1}{1+r_i^2-2r_i\cos\theta},\, \text{for}\, i=1,2.
   \end{equation}
By replacing 
     \begin{equation}\nonumber
        1+r_i^2-2r_i\cos\theta=\frac{tr_i\sin\theta}{\theta}
     \end{equation}
into (\ref{eq:4.19}), to prove our claim, it is enough to prove that
    \begin{equation}\label{eq:4.20}
      r_2r_1^2-r_2+r_1r_2^2-r_1=0.
    \end{equation}
Since $r_1r_2=1$, (\ref{eq:4.20}) is true and this finishes the proof.
\end{proof}
\begin{prop}\label{prop:4.14}
For $t>0$, let $s_t(x)$ be the density of the measure $\sigma_t$ at $x\in(0,+\infty)$.
\begin{enumerate}[$(1)$]
  \item The support of the function $s_t$ is the interval $[x_3(t),x_4(t)]$.
  For any $x\in(0,+\infty)$, we have that 
     \begin{equation}\nonumber
       s_t(x)x=s_t\left(\frac{1}{x}\right)\frac{1}{x}. 
     \end{equation}
     Moreover, the function $x\rightarrow s_t(x)x$ is strictly increasing on the interval $(x_3(t),1]$.
  \item The measure $\sigma_t$ is invariant under the map of $\mathbb{R}:x\rightarrow 1/x$. In other words,
  for any interval $[\alpha,\beta]\subset (0,1]$, we have that
      \begin{equation}\nonumber
        \sigma_t([\alpha,\beta])=\sigma_t\left(\left[\frac{1}{\beta}, \frac{1}{\alpha} \right] \right).
      \end{equation}
  \item The function $s_t$ is a strictly decreasing function of $x$ on the interval $[1,x_4(t))$.
\end{enumerate} 
\end{prop}
\begin{proof}
From Theorem \ref{thm:4.6} and Lemma \ref{lemma:4.13}, we have that 
       \begin{equation}\label{eq:4.23}
         s_t\left(\frac{1}{\Lambda_t(r)} \right)=\Lambda_t(r)\frac{1}{\pi}\frac{u_t(r)}{t},
       \end{equation}
and that 
       \begin{equation}\label{eq:4.25}
         \begin{split}
            s_t(\Lambda_t(r))&=s_t\left(\frac{1}{\Lambda_t\left(\frac{1}{r} \right)} \right)\\
                   &=\frac{1}{\pi}\Lambda_t\left(\frac{1}{r} \right)\frac{u_t\left(\frac{1}{r}\right)}{r}\\
                   &=\frac{1}{\pi}\frac{1}{\Lambda_t(r)}\frac{u_t(r)}{t}.
         \end{split}
       \end{equation}
Note that $\Lambda_t:(0,+\infty)\rightarrow (0,+\infty)$ is a homeomorphism and $\Lambda_t$ is increasing. We also have 
that $\Lambda_t(V_t)=\Lambda_t((x_1(t),x_2(t)))=(x_3(t),x_4(t))$, then (1) follows from Proposition \ref{prop:4.12},
Lemma \ref{lemma:4.13} and 
(\ref{eq:4.25}). Part (2) 
is a consequence of (1). Since $u_t$ is an increasing function of $r$ on the interval $(x_1(t),1]$ 
by Proposition \ref{prop:4.12}, then $\Lambda_t u_t$ is also an increasing function of $r$ on $(x_1(t),1]$ and
Part (3) follows from this fact and (\ref{eq:4.23}).
\end{proof}
\begin{rmk}
\emph{Part (2) of Proposition \ref{prop:4.14} is not unexpected. In fact,
for any integer $n$, let $a_n=1+\sqrt{2t/n}$, $b_n=1/a_n$ and $\mu_n=(\delta_{a_n}+\delta_{b_n})/2$,
then from \cite[Lemma 7.1]{BV1992}, we see that $\mu_n^{\boxtimes n} \rightarrow \sigma_t$ weakly as 
$n\rightarrow +\infty$. 
}
\end{rmk}
\begin{rmk}\label{rmk:order2}
\emph{
 By a similar calculation in Remark \ref{rmk:order1}, we can see that there exist $c(t)>0, d(t)>0$ such that
     \begin{equation}\label{eq:200}
        s_t(x)=c(t)|x-x_3(t)|^{\frac{1}{2}}(1+o(1))
     \end{equation}
in a small interval $[x_3(t),x_3(t)+\delta)$ for some $\delta>0$, and that
    \begin{equation}\label{eq:202}
        s_t(x)=d(t)|x-x_4(t)|^{\frac{1}{2}}(1+o(1))
    \end{equation}
in a small interval $(x_4(t)-\delta',x_4(t)]$ for some $\delta'>0$. 
The orders in (\ref{eq:200}) and (\ref{eq:202}) are essentially due to the fact that
$H_t$ has zeros of order one at $x_1(t)$ and $x_2(t)$.
}
\end{rmk}

   \begin{figure}[H]
     \includegraphics[scale=0.9]{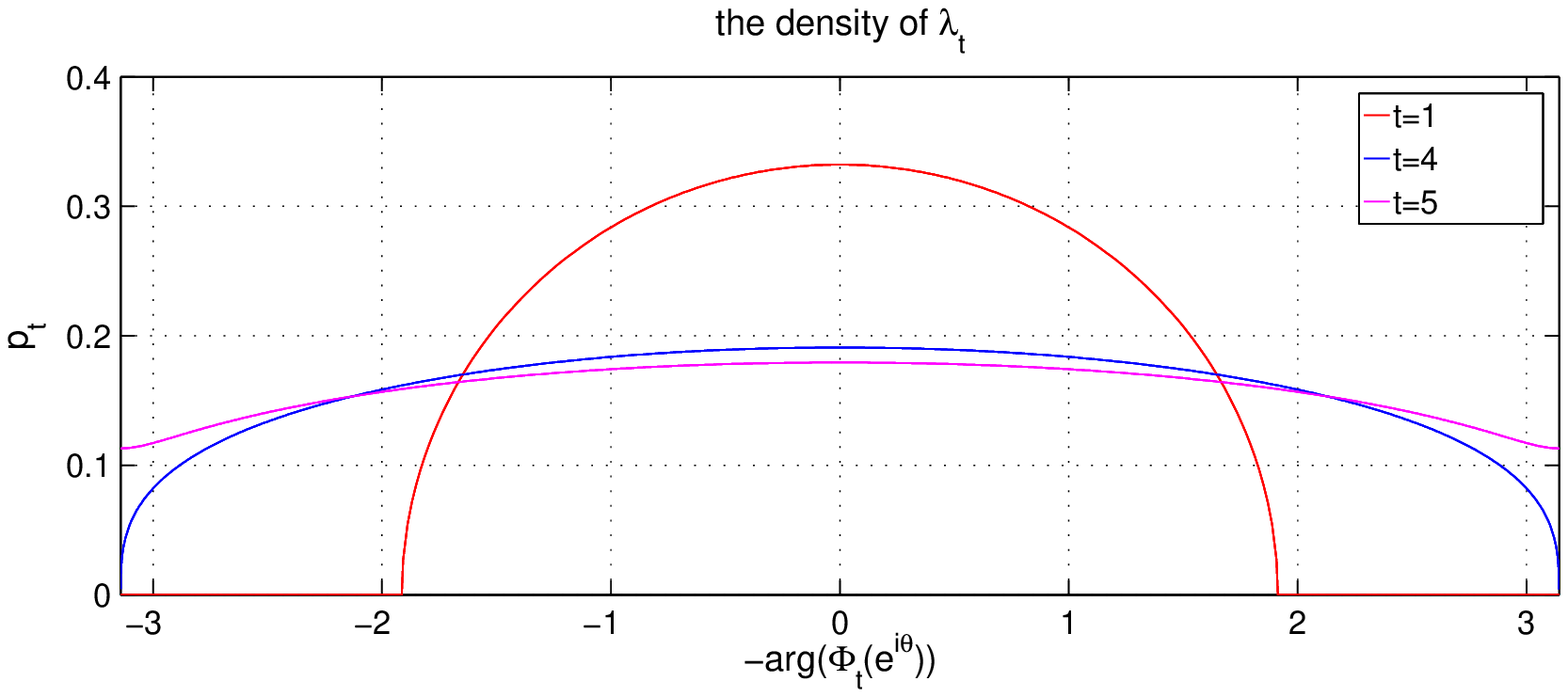}
    \vspace{0pt}
    \label{fig1}
   \end{figure}
   
      \begin{figure}[H]
     \includegraphics[scale=1.2]{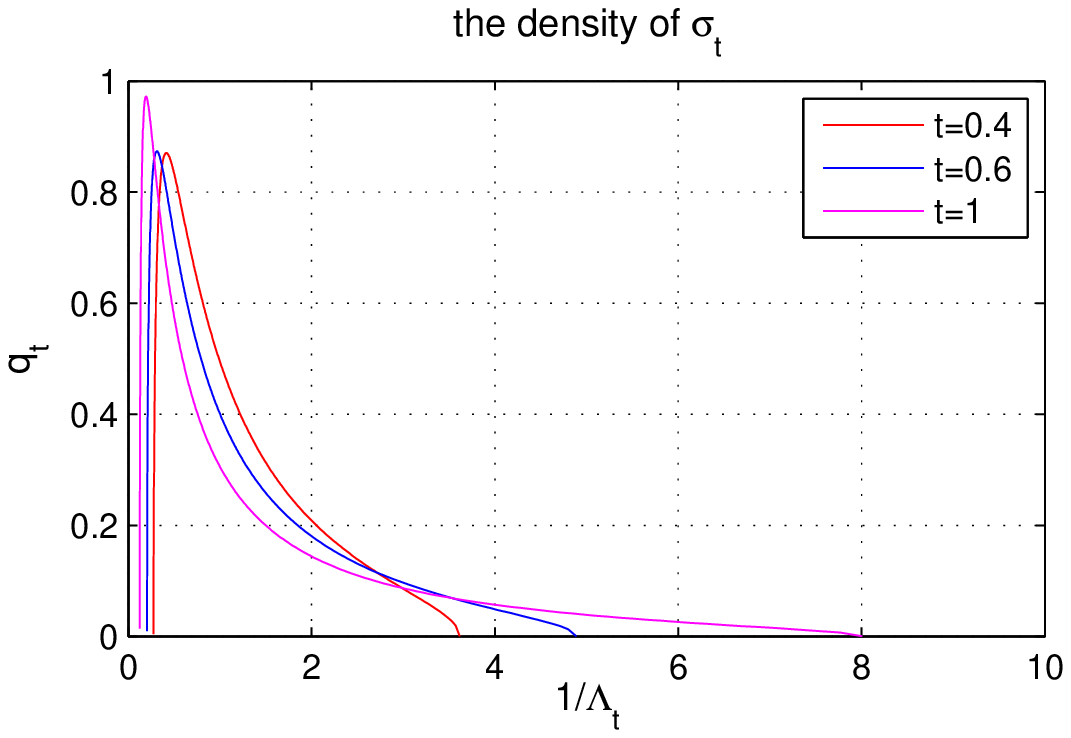}
    \vspace{0pt}
    \label{fig2}
   \end{figure}
\section*{Acknowledgements}
The author would like to express his sincere gratitude to his advisor, 
Professor Hari Bercovici, for constant help and many valuable discussions.
He also wants to thank Professors Michael Anshelevich, Todd Kemp for comments 
and Honghu Liu for his help with Matlab. He is grateful to a referee for a
very careful reading and lots of suggestions.

\bibliography{clipboard20140314}

\providecommand{\bysame}{\leavevmode\hbox to3em{\hrulefill}\thinspace}
\providecommand{\MR}{\relax\ifhmode\unskip\space\fi MR }
\providecommand{\MRhref}[2]{%
  \href{http://www.ams.org/mathscinet-getitem?mr=#1}{#2}
}
\providecommand{\href}[2]{#2}
\begin{thebibliography}{10}

\bibitem{Belinschi06}
Serban~T. Belinschi, \emph{A note on regularity for free convolutions}, Ann.
  Inst. H. Poincar\'e Probab. Statist. \textbf{42} (2006), no.~5, 635--648.
  \MR{2259979 (2007g:46098)}

\bibitem{BB2005}
Serban~T. Belinschi and Hari Bercovici, \emph{Partially defined semigroups
  relative to multiplicative free convolution}, Int. Math. Res. Not. (2005),
  no.~2, 65--101. \MR{2128863 (2006f:46061)}

\bibitem{BB2007new}
\bysame, \emph{A new approach to subordination results in free probability}, J.
  Anal. Math. \textbf{101} (2007), 357--365. \MR{2346550 (2008i:46059)}

\bibitem{BV1992}
Hari Bercovici and Dan Voiculescu, \emph{L\'evy-{H}in\v cin type theorems for
  multiplicative and additive free convolution}, Pacific J. Math. \textbf{153}
  (1992), no.~2, 217--248. \MR{1151559 (93k:46052)}

\bibitem{BV1993}
\bysame, \emph{Free convolution of measures with unbounded support}, Indiana
  Univ. Math. J. \textbf{42} (1993), no.~3, 733--773. \MR{1254116 (95c:46109)}

\bibitem{BianeBM}
Philippe Biane, \emph{Free {B}rownian motion, free stochastic calculus and
  random matrices},  \textbf{12} (1997), 1--19. \MR{1426833 (97m:46104)}

\bibitem{Biane1997}
\bysame, \emph{On the free convolution with a semi-circular distribution},
  Indiana Univ. Math. J. \textbf{46} (1997), no.~3, 705--718. \MR{1488333
  (99e:46084)}

\bibitem{BianeJFA}
\bysame, \emph{Segal-{B}argmann transform, functional calculus on matrix spaces
  and the theory of semi-circular and circular systems}, J. Funct. Anal.
  \textbf{144} (1997), no.~1, 232--286. \MR{1430721 (97k:22011)}

\bibitem{Biane1998}
\bysame, \emph{Processes with free increments}, Math. Z. \textbf{227} (1998),
  no.~1, 143--174. \MR{1605393 (99e:46085)}

\bibitem{CDDF}
Mireille Capitaine, Catherine Donati-Martin, Delphine Féral, and Maxime
  Février, \emph{Free convolution with a semicircular distribution and
  eigenvalues of spiked deformations of wigner matrices}, Electron. J. Probab.
  \textbf{16} (2011), no. 64, 1750--1792.

\bibitem{2013CG}
Guillaume C{\'e}bron, \emph{Free convolution operators and free {H}all
  transform}, J. Funct. Anal. \textbf{265} (2013), no.~11, 2645--2708.
  \MR{3096986}

\bibitem{CK2014}
Beno{\^{\i}}t Collins and Todd Kemp, \emph{Liberation of projections}, J.
  Funct. Anal. \textbf{266} (2014), no.~4, 1988--2052. \MR{3150150}

\bibitem{DH2011}
Nizar Demni and Taoufik Hmidi, \emph{Spectral distribution of the free unitary
  {B}rownian motion: another approach}, S\'eminaire de {P}robabilit\'es {XLIV},
  Lecture Notes in Math., vol. 2046, Springer, Heidelberg, 2012, pp.~191--206.
  \MR{2953348}

\bibitem{2013DHK}
Bruce~K. Driver, Brian~C. Hall, and Todd Kemp, \emph{The large-{$N$} limit of
  the {S}egal-{B}argmann transform on {$\Bbb{U}_N$}}, J. Funct. Anal.
  \textbf{265} (2013), no.~11, 2585--2644. \MR{3096985}

\bibitem{HZ}
Hao-Wei Huang and Ping Zhong, \emph{On the support of measures in
  multiplicative free convolution semigroups}, arXiv:1302.4466 math.CV; Math.
  Z., to appear (2013).

\bibitem{IU2013}
Masaki Izumi and Yoshimichi Ueda, \emph{{Remarks on free mutual information and
  orbital free entropy}}, arXiv:1306.5372 math.OA (2013).

\bibitem{2013Kemp2}
Todd Kemp, \emph{{Heat Kernel Empirical Laws on $\mathbb{U}_N$ and
  $\mathbb{GL}_N$}}, arXiv:1306.2140, math.PR (2013).

\bibitem{2013Kemp1}
\bysame, \emph{{The Large-$N$ Limits of Brownian Motions on $\mathbb{GL}_N$}},
  arXiv:1306.6033 math.PR (2013).

\bibitem{Rains}
Eric~M. Rains, \emph{Combinatorial properties of {B}rownian motion on the
  compact classical groups}, J. Theoret. Probab. \textbf{10} (1997), no.~3,
  659--679. \MR{1468398 (99f:60016)}

\bibitem{Singer}
Isadore~M. Singer, \emph{On the master field in two dimensions}, Functional
  analysis on the eve of the 21st century, {V}ol.\ 1 ({N}ew {B}runswick, {NJ},
  1993), Progr. Math., vol. 131, Birkh\"auser Boston, Boston, MA, 1995,
  pp.~263--281. \MR{1373007 (96m:81208)}

\bibitem{DVV1987}
Dan Voiculescu, \emph{Multiplication of certain noncommuting random variables},
  J. Operator Theory \textbf{18} (1987), no.~2, 223--235. \MR{915507
  (89b:46076)}

\bibitem{DVV1999}
\bysame, \emph{The analogues of entropy and of {F}isher's information measure
  in free probability theory. {VI}. {L}iberation and mutual free information},
  Adv. Math. \textbf{146} (1999), no.~2, 101--166. \MR{1711843 (2001a:46064)}

\bibitem{Basic}
Dan Voiculescu, Ken Dykema, and Alexandru Nica, \emph{Free random variables},
  CRM Monograph Series, vol.~1, American Mathematical Society, Providence, RI,
  1992, A noncommutative probability approach to free products with
  applications to random matrices, operator algebras and harmonic analysis on
  free groups. \MR{1217253 (94c:46133)}

\bibitem{Zhong1}
Ping Zhong, \emph{Free {B}rownian motion and free convolution semigroups:
  multiplicative case}, arXiv:1210.6090 math.FA, PR; Pacific J. Math., to
  appear (2012).

\end{thebibliography}
\bibliographystyle{amsplain}

\end{document}